\documentclass{amsart}

\usepackage{amsmath}
\usepackage{amssymb}
\usepackage{color}
\usepackage{amsthm}
\usepackage{tikz}
\usepackage{tikz-cd}
\usetikzlibrary{positioning,calc,cd}
\usepackage{hyperref}
\usepackage{tensor}
\usepackage{relsize}
\usepackage{mathtools}
\usepackage{quiver}

\newtheorem{theorem}{Theorem}[section]
\newtheorem{proposition}[theorem]{Proposition}
\newtheorem{lemma}[theorem]{Lemma}
\newtheorem{corollary}[theorem]{Corollary}
\newtheorem{definition}[theorem]{Definition}
\newtheorem{example}[theorem]{Example}

\definecolor{Noir}{rgb}{0,0,0} 
\definecolor{Blanc}{rgb}{1,1,1} 
\definecolor{Gray}{rgb}{0.5,0.5,0.5} 
\definecolor{Rouge}{rgb}{0.8,0.1,0.1} 
\definecolor{DBleu}{RGB}{51,51,178} 
\definecolor{LBleu}{rgb}{0.85,0.85,1} 
\definecolor{Orange}{RGB}{255,140,0} 
 
\newcommand{\bcent}{\begin{center}} 
\newcommand{\ecent}{\end{center}} 
 
\newcommand{\benum}{\begin{enumerate}} 
\newcommand{\eenum}{\end{enumerate}} 
\newcommand{\bitem}{\begin{itemize}} 
\newcommand{\eitem}{\end{itemize}} 
 
\newcommand{\btab}{\begin{tabular}} 
\newcommand{\etab}{\end{tabular}} 
\newcommand{\beqn}{\begin{eqnarray}} 
\newcommand{\eeqn}{\end{eqnarray}} 
 
\newcommand{\bmath}{\begin{math}} 
\newcommand{\emath}{\end{math}} 
 
\newcommand{\noin}{\noindent} 

\newcommand{\CC}{\mathbb C}
\newcommand{\DD}{\mathbb D}
\newcommand{\HH}{\mathbb H}

\newcommand{\PP}{\mathbb P}
\newcommand{\RR}{\mathbb R}
\newcommand{\ZZ}{\mathbb Z}
\newcommand{\QQ}{\mathbb Q}

\newcommand{\frakD}{\mathfrak D}

\newcommand{\frakk}{\mathfrak k}

\newcommand{\calG}{\mathcal G}
\newcommand{\calP}{\mathcal P}

\newcommand{\calF}{\mathcal F}

\newcommand{\calT}{\mathcal T}

\newcommand{\uQQ}{\underline{\mathbb{Q}}}

\newcommand{\ssslash}{\mathbin{/\mkern-6mu/\mkern-6mu/}}

\newcommand{\loopquiver}{%
    \begin{tikzpicture}[scale=0.35, baseline=-1.75mm, thick]
    \node (a) at (0,0) {$\bullet$};
\draw (0,-0.3) circle (0.4cm);
    \end{tikzpicture}%
}

\newcommand{\triquiver}{%
 \begin{tikzpicture}[scale = 0.44, baseline = -1mm,]
\node(a) at (0,-0.04) {$\bullet$};
\node(b) at (2,-0.04) {$\bullet$};

\draw[-] (0,0) -- (2,0);
\draw[-] (0,0) to [out = 30, in =150] (2,0);
\draw[-] (0,0) to [out = 330, in = 210] (2,0);
\end{tikzpicture}%
}

\newcommand{\diquiver}{%
 \begin{tikzpicture}[scale = 0.44, baseline = -1mm,]
\node(a) at (0,-0.04) {$\bullet$};
\node(b) at (2,-0.04) {$\bullet$};

\draw[-] (0,0) to [out = 30, in =150] (2,0);
\draw[-] (0,0) to [out = 330, in = 210] (2,0);
\end{tikzpicture}%
}

\newcommand{\fourquiver}{%
 \begin{tikzpicture}[scale = 0.44, baseline = -1mm,]
\node(a) at (0,-0.04) {$\bullet$};
\node(b) at (2,-0.04) {$\bullet$};

\draw[-] (0,0) to [out = 15, in =165] (2,0);
\draw[-] (0,0) to [out = 345, in = 195] (2,0);
\draw[-] (0,0) to [out = 40, in =140] (2,0);
\draw[-] (0,0) to [out = 320, in = 220] (2,0);
\end{tikzpicture}%
}

\newcommand{\diquiverplus}{%
 \begin{tikzpicture}[scale = 0.44, baseline = -1mm,]
\node(a) at (0,-0.04) {$\bullet$};
\node(b) at (2,-0.04) {$\bullet$};
\node(c) at (3,-0.04) {$\bullet$};

\draw[-] (0,0) to [out = 30, in =150] (2,0);
\draw[-] (0,0) to [out = 330, in = 210] (2,0);
\draw[-] (2,0) to (3,0);
\end{tikzpicture}%
}

\newcommand{\newquiver}{%
 \begin{tikzpicture}[scale = 0.44, baseline = -1mm,]
\node(a) at (0,-0.04) {$\bullet$};
\node(b) at (2,-0.04) {$\bullet$};
\node (c) at (1,0.75){$\bullet$};

\draw[-] (0,0) -- (1,0.8) -- (2,0);
\draw[-] (0,0) to [out = 20, in =160] (2,0);
\draw[-] (0,0) to [out = 340, in = 200] (2,0);
\end{tikzpicture}%
}

\renewcommand{\epsilon}{\varepsilon}

\providecommand{\v}[1]{\vec{#1}}

\newtheorem{manualtheoreminner}{Theorem}
\newenvironment{manualtheorem}[1]{%
  \renewcommand\themanualtheoreminner{#1}%
  \manualtheoreminner
}{\endmanualtheoreminner}

\newcommand\varleq{\mathbin{\vcenter{\hbox{%
  \oalign{\hfil$\scriptstyle<$\hfil\cr 
          \noalign{\kern-.3ex} 
          $\scriptscriptstyle({-})$\cr}%
}}}} 
 
\renewcommand\subsetneq{\mathbin{\vcenter{\hbox{%
  \oalign{\hfil$\scriptstyle\subset$\hfil\cr 
          \noalign{\kern-.3ex} 
          $\scriptscriptstyle({-})$\cr}%
}}}} 

\newtheorem{manualcorollaryinner}{Corollary}
\newenvironment{manualcorollary}[1]{%
  \renewcommand\themanualcorollaryinner{#1}%
  \manualcorollaryinner
}{\endmanualcorollaryinner}

\renewcommand\subsetneq{\mathbin{\vcenter{\hbox{%
  \oalign{\hfil$\scriptstyle\subset$\hfil\cr 
          \noalign{\kern-.3ex} 
          $\scriptscriptstyle({-})$\cr}%
}}}} 

\author{Evan Sundbo}
\address{Department of Pure Mathematics, University of Waterloo\\ Waterloo ON, Canada}
\email{e2sundbo@uwaterloo.ca}

\title[]{Balloon Animal Maps with Applications to the Cohomology of Hypertoric Hitchin Systems}
\date{\today}
\subjclass[2010]{}

\begin{document}

\maketitle

\begin{abstract}
In this article we introduce the notion of a balloon animal map between broken toric varieties and construct several long exact sequences in cohomology related to them. We give a new proof of the deletion-contraction relation on hypertoric Hitchin systems of Dansco--Mcbreen--Shende and present some refinements of it. The end result is a formula for the Poincar\'e polynomial of any hypertoric Hitchin system associated to a graph with first Betti number $2$ along with a recipe to calculate the Poincar\'e polynomial of any hypertoric Hitchin system, given knowledge of a finite number of base cases.
\end{abstract}

\setcounter{tocdepth}{1}
\tableofcontents
\section{Introduction}\label{intro}

A broken toric variety is a union of smooth projective toric varieties glued to each other along torus-invariant toric subvarieties. With the aim of understanding the rational cohomology of broken toric varieties, the author proves in \cite{Sun25} a decomposition theorem for the derived pushforward of the constant sheaf of a broken toric variety on its polytopal complex. As a corollary, one finds a canonical splitting
$$
H^k(X,\underline{\QQ}_X)\cong \bigoplus_{i+j=k}H^j(P_\bullet, R^if_*\uQQ_X).
$$ 

The problem is thus reduced to calculating the cohomology of the higher derived pushforward sheaves $R^if_*\uQQ_X$ on polytopal complexes of broken toric varieties. This is more tractable but not trivial. In \cite{Sun25} some vanishing theorems and other tools are developed which allow one to attack the problem, yielding formulas for the Betti numbers in some classes of examples (including when the polytopal complex is contractible or is the skeleton of a higher-dimensional polytope).

In this article we approach the problem by studying \emph{balloon animal maps} between broken toric varieties and thus relating the cohomology groups of their higher derived pushforwards.  Given a broken toric variety $X$ and a subdivision of its polytopal complex, one constructs two new broken toric varieties $X^{\cup}$ and $X^\cap$ (along with polytopal complexes $P_\bullet^\cup$ and $P_\bullet^\cap$). The balloon animal map $b$ goes from $X$ to $X^\cup$ and induces maps in cohomology that allow us to compare information about the cohomology groups of $X$, $X^\cup$, and $X^\cap$.

As highlights we have
\begin{manualtheorem}{(Lemma \ref{lem_bam} (2) and Theorem \ref{bigballoonpf})}[]
Let $b:X\to X^{\cup}$ be a balloon animal map between broken toric varieties induced by a generic subdivision of the polytopal complex of $X$ and let $X^{\cap}$ be the broken toric subvariety of $X^\cup$ associated to $P^{\cap}_\bullet$. There is then a distinguished triangle in the derived category of sheaves on $X^{\cup}$
$$\underline{\QQ}_{X^{\cup}}\to Rb_*\underline{\QQ}_X\to\underline{\QQ}_{X^\cap}[-1]\xrightarrow[]{+1}$$
and an associated long exact sequence
\begin{equation*}
\dots\to H^{k}(X^\cup,\underline{\QQ}_{X^{\cup}})\to H^k(X,\underline{\QQ}_X)\to H^{k-1}(X^\cap,\underline{\QQ}_{X^\cap})\to\dots.
\end{equation*}
\end{manualtheorem}

along with the more refined statement at the level of polytopal complexes

\begin{manualtheorem}{(Lemma \ref{bampushforward} and Theorem \ref{balloonpf})}
Let $b:X\to X^{\cup}$ be a balloon animal map between broken toric varieties induced by a generic subdivision of the polytopal complex of $X$ and let $X^{\cap}$ be the broken toric subvariety of $X^\cup$ associated to $P^{\cap}_\bullet$.  For $i\geq 0$ there exist short exact sequences
\begin{equation*}
0 \to R^if_*\underline{\QQ}_{X^\cup}\to  R^if_*\underline{\QQ}_{X}\to  R^{i-1}f_*\underline{\QQ}_{X^\cap}\to 0.
\end{equation*} and associated long exact sequences
\begin{equation*}
\dots\to H^{k}(P_\bullet^{X^\cup},R^if\underline{\QQ}_{X^{\cup}})\to H^k(P_\bullet,R^if\underline{\QQ}_X)\to H^{k}(P_\bullet^{X^\cap},R^{i-1}f\underline{\QQ}_{X^\cap})\to\dots.
\end{equation*}
\end{manualtheorem}

In Section \ref{section:cellular} we turn to more concrete calculations. By noting that the purely combinatorial nature of the sheaves $R^if_*\uQQ_X$ allows us to view them as cellular sheaves, we can use elementary algebraic topology methods to calculate the cohomology groups by hand in simple enough examples. The philosophy is then that this information can be bootstrapped to more complicated examples via the long exact sequences coming from balloon animal maps. This approach is carried out in the last section for a particular class of broken toric varieties as we will soon explain.

Broken toric varieties appear in areas of geometry as varied as mirror symmetry, hypertoric geometry, toric topology, and the study of compactified Jacobians and Hitchin systems. We refer to  \cite{Sun25} for a more dedicated discussion of their history and ubiquity, but let us mention in particular their recent appearance in Section 7.3 of \cite{Low24}, wherein broken toric varieties are used as a tool to give an example of a singular proper toric variety over a finite field with nontrivial negative equivariant $K$-theory.

The broken toric varieties on which we focus are those appearing in the study of hypertoric Hitchin systems. Let us begin here with some motivation by reviewing some of the geometry of the moduli space of Higgs bundles and explain its relationship with the hypertoric Hitchin systems. 
 
 \begin{definition}
 Let $X$ be a compact Riemann surface of genus $g\geq 2$. A \emph{Higgs bundle} on $X$ consists of a holomorphic vector bundle $E$ on $X$ along with a twisted endomorphism $\phi:E\to E\otimes K_X$. A Higgs bundle is \emph{stable} if $\frac{\text{deg}F}{\text{rk}F}<\frac{\text{deg}E}{\text{rk}E}$ for all sub-bundles $F\subset E$ for which $\phi(F) \subseteq F\otimes K_X$.
 \end{definition}
  The moduli space  $\mathcal{M}_X(r,d)$ of rank $r$ degree $d$ stable Higgs bundles on $X$ is called the \emph{Hitchin system}. These moduli spaces have more fascinating structures than one can shake a stick at: they are hyperk\"ahler completely integrable systems which come equipped with a proper morphism, the \emph{Hitchin map} $h$, to an affine space (the space of spectral curves) whose generic fibres are abelian varieties. Through the non-Abelian Hodge theorem there is a one-to-one correspondence between stable Higgs bundles on $X$, flat connections on $X$, and $\text{GL}_r$-representations of the fundamental group of $X$.  These three moduli spaces are sometimes called the \emph{Dolbeault}, \emph{de Rham}, and \emph{Betti} moduli spaces, respectively.

Importantly for our purposes, the BNR Correspondence \cite{BNR89} states that for a smooth spectral curve $\Sigma$, the fibre $h^{-1}(\Sigma)$ is the Jacobian variety of $\Sigma$.  This also works for nodal spectral curves, where the fibre is a fine compactified Jacobian. Note that \emph{a priori} there is more than one way to compactify a Jacobian, depending on a stability condition, but the cohomology of the compactification is independent of the chosen (generic) stability condition by \cite{MSV21}. Further, it follows from Proposition 7.5.1 in \cite{N10} that $$H^\bullet(\overline{\text{Jac}(\Sigma)})\cong H^\bullet(\text{Jac}(\tilde{\Sigma}))\otimes D(\Sigma),$$
where $\tilde{\Sigma}$ is the normalization of $\Sigma$ and $D(\Gamma_\Sigma)$ is some vector space depending only on the dual graph\footnote{The \emph{dual graph} of a nodal curve $\Sigma$ consists of a vertex for every component of $\Sigma$ and an edge between two vertices for every node between their corresponding components.} $\Gamma_\Sigma$ of $\Sigma$.

In the case that all of the components of $\Sigma$ are rational, $H^\bullet(\overline{\text{Jac}(\Sigma)})\cong D(\Sigma)$ since the Jacobians that would be in play are trivial. We can alternatively describe the compactified Jacobian as a broken toric variety.  The Jacobian $\text{Jac}(\Sigma) \cong (\CC^*)^d$ (where $d$ is the number of components of $X$) acts on $\overline{\text{Jac}(\Sigma)}$ by tensoring, and as a subvariety it is dense by construction.  Taking the quotient by the compact part $U(1)^d$ of $(\CC^*)^d$ yields a periodic hyperplane arrangement as described in \cite{OS79}. 

On the other hand, in an unpublished note from 2015 Hausel and Proudfoot described a special kind of hypertoric variety associated to a periodic hyperplane arrangement in $\RR^n$. These varieties have three incarnations suggestively named the Dolbeault $\mathfrak{D}$, Betti $\mathfrak{B}$, and de Rham $\mathfrak{dR}$ hypertoric varieties.  Together they are known as \emph{multiplicative hypertoric varieties} or (due to subsequent authors) \emph{Hausel-Proudfoot varieties}.  The motivation for studying such objects comes from several sources.  First, the multiplicative hypertoric varieties really do arise naturally as multiplicative versions of the additive hypertoric varieties of Bielawski--Dancer \cite{BD00} and Konno \cite{Ko00}.  Second, they can be thought of as a generalization of the multiplicative quiver constructions of Crawley-Boevey--Shaw \cite{CBS06}, objects which are themselves currently attracting some attention, e.g. \cite{Jo14, Ma19}.

Our attention is on the Dolbeault multiplicative hypertoric varieties (or \emph{hypertoric Hitchin systems}). They come equipped with a proper morphism $h$ to an affine space and indeed with the structure of an integrable system. The generic fibres of this morphism are abelian varieties and there is a deformation retract to the central fibre which turns out to be broken toric. One constructs the hypertoric Hitchin systems by starting with the simplest possible example $\frakD(\loopquiver)$ and then obtaining others as quasi-hyperk\"ahler reductions of products of this ``base space", controlled by the further ingredient of a connected graph $\Gamma$. The hypertoric Hitchin system associated to the graph $\Gamma$ is denoted $\frakD(\Gamma)$.

The periodic hyperplane arrangement arising from a compacitified Jacobian $\overline{\text{Jac}(\Sigma)}$ is the same periodic hyperplane arrangement which describes the polytopal complex of the broken toric central fibre of $\frakD(\Gamma_\Sigma)$, as we will describe in Subsection \ref{props}.  In particular, this shows that in the case of all rational components, $H^\bullet(\overline{\text{Jac}(\Sigma)})\cong  H^\bullet (\frakD(\Gamma_\Sigma))$. It is true for any nodal spectral curve that $D(\Sigma) \cong H^\bullet(\frakD(\Gamma_\Sigma))$ by Remark 10.3.5 of \cite{DMS24}. In this sense, the hypertoric Hitchin systems are combinatorial toy models for the cohomology around singular fibres of the Hitchin system.

Applying the tools of balloon animal maps to the broken toric varieties appearing in hypertoric Hitchin systems allows us to give a new proof of the following result of Dansco--McBreen--Shende.

\begin{manualtheorem}{\ref{thm_delcon} (Theorem 6.43 of \cite{DMS24})}
Let $\Gamma$ be a graph and $e$ be an edge of $\Gamma$ which is not a loop. Then there is a long exact sequence relating the cohomologies of hypertoric Hitchin systems:
\begin{equation*}
\dots\to H^{k}(\frakD(\Gamma),\underline{\QQ}_{\frakD(\Gamma)})\to H^k(\frakD(\Gamma/e),\underline{\QQ}_{\frakD(\Gamma/e)})\to H^{k-1}(\frakD(\Gamma \setminus e),\underline{\QQ}_{\frakD(\Gamma \setminus e)})\to\dots
\end{equation*}
\end{manualtheorem}

 We can also refine this to new statements at the level of polytopal complexes.
 
 \begin{manualtheorem}{\ref{dcpf}}
Let $\Gamma$ be a graph and $e$ be an edge of $\Gamma$ which is not a loop. Then there is a long exact sequence
\begin{equation*}
\dots\to H^{k}(T^n, R^{i}f_*\underline{\QQ}_{\frakD(\Gamma)})\to H^{k}(T^n, R^{i}f_*\underline{\QQ}_{\frakD(\Gamma/ e)})\to H^{k}(T^{n-1}, R^{i-1}f_*\underline{\QQ}_{\frakD(\Gamma\setminus e)}) \to\dots
\end{equation*}
\end{manualtheorem}

 \begin{manualtheorem}{\ref{smalldelcon}}
Let $\Gamma$ be a graph with $b_1(\Gamma) =n \geq 2$ and $e$ an edge of $\Gamma$ which is incident to a vertex of degree $2$. Then 
\begin{equation*}
H^k(T^n, R^if_*\underline{\QQ}_{\Gamma}) \cong H^k(T^n, R^if_*\underline{\QQ}_{\Gamma / e}) \oplus H^{k-1}(T^n, R^{i-1}f_*\underline{\QQ}_{\Gamma\setminus e}).
\end{equation*}
\end{manualtheorem}

Theorem \ref{smalldelcon} is especially useful for the following reason. If we are interested in calculating the cohomology groups of $\frakD(\Gamma)$ with $b_1(\Gamma)=n$ and no disconnecting vertices\footnote{ Let $\Gamma$ be a connected graph. The \emph{induced subgraph} of a subset of edges $E'\subset E(\Gamma)$ is the subgraph of $\Gamma$ containing only the edges of $E'$ and all vertices which those edges are incident to. A vertex $v$ of $\Gamma$ is a \emph{disconnecting vertex} if the set of edges of $E(\Gamma)$ of $\Gamma$ can be partitioned into two subsets $E_1,E_2$ such that the intersection of the induced subgraphs of $E_1$ and $E_2$ is $\{v\}$.}, it is enough to know the cohomology groups of $\frakD(\Gamma')$ for all graphs $\Gamma'$ with $b_1(\Gamma)\leq n$, no disconnecting vertices, and no vertices of degree $2$. This idea is made precise in Corollary \ref{intermsofbase}.

Disconnecting vertices in $\Gamma$ are also no real obstacle to calculating cohomology groups:

\begin{manualcorollary}{\ref{disconcorr}}
Let $v$ be a disconnecting vertex of a connected graph $\Gamma$ with $b_1(\Gamma)=n$ so that $\Gamma$ can be written as the union of a finite number of connected graphs $\Gamma_1,\ldots,\Gamma_m$ with $\bigcap_{i=1}^m\Gamma_i = \{v\}$.  Then the Poincar\'e polynomial of $\frakD(\Gamma)$ is the product of the Poincar\'e polynomials of $\frakD(\Gamma_i)$, or equivalently
$$H^k(\frakD(\Gamma),\underline{\QQ}_{\frakD(\Gamma)}) \cong \bigoplus_{k_1+\dots+k_m=k}\left(\bigotimes_{i=1}^m H^{k_i}(\frakD(\Gamma_i),\underline{\QQ}_{\frakD(\Gamma_i)})\right)$$
\end{manualcorollary}

Bringing this all together, we end with a closed form expression for the Betti numbers of the hypertoric Hitchin system associated to a graph $\Gamma$ with first Betti number $2$.

\begin{manualtheorem}{\ref{dim2}}
Let $\Gamma$ be a graph with $b_1(\Gamma) = 2$ and $t(\Gamma)$ be the number of spanning trees of $\Gamma$. Then
\begin{enumerate}
\item If $\Gamma$ has a disconnecting vertex, then the Poincar\'e polynomial of $\mathfrak{D}(\Gamma)$ is 
$$1+2y+(|E(\Gamma)|+1)y^2+|E(\Gamma)|y^3+t(\Gamma)y^4.$$
\item If $\Gamma$ does not have a disconnecting vertex, then the Poincar\'e polynomial of $\mathfrak{D}(\Gamma)$ is 
$$1+2y+|E(\Gamma)|y^2+(|E(\Gamma)|-1)y^3+t(\Gamma)y^4.$$
\end{enumerate}

\end{manualtheorem}

\noin\emph{Acknowledgements.} The author would like to thank Michael McBreen for interesting discussions regarding this work and Jakub Löwit for pointing out the appearance of broken toric varieties in his work in equivariant $K$-theory.  The auther also especially thanks Michael Groechenig for introducing him to this problem and for guiding him through working on it as part of his doctoral studies. This work was partially supported by a Malcolm Slingsby Robertson Fellowship in Mathematics at the University of Toronto.\\

\section{Broken Toric Varieties}\label{btv}

\begin{definition}\label{btvdef} A \emph{broken toric variety} is a reducible algebraic variety whose components are equidimensional smooth projective toric varieties whose pairwise intersections are torus-invariant closed toric subvarieties of each component. The \emph{polytopal complex} $P_\bullet$ of a broken toric variety $X$ is the union of the polytopes of the components of $X$ glued to each other so that if two components of X meet along a toric variety $X'$, then the polytopes of those components intersect as the polytope of $X'$. 
\end{definition}

Note that to any complex of polytopes $P_\bullet$ there may be more than one isomorphism class of broken toric varieties which have $P_\bullet$ as their associated polytopal complex. Given $P_\bullet$, a broken toric variety $X$ with components $X_i$ is fixed by choosing gluing isomorphisms $\alpha_{ij}$ from a toric subvariety in $X_i$ to an isomorphic subvariety in $X_j$ (satisfying the cocycle condition $\alpha_{ij}\alpha_{jk} = \alpha_{ik}$). Such gluing data determines an element of $H^1(P_\bullet,T^n_\CC)$, a $T^n$-torsor on $P_\bullet$. We further note that there is not necessarily a bijection of $H^1(P_\bullet,T^n_\CC)$ with the space of isomorphism classes of broken toric varieties over $P_\bullet$\textemdash{}that is to say, different gluing isomorphisms may lead to isomorphic (as algebraic varieties) broken toric varieties.  This leads directly to the following lemma.

\begin{lemma}\label{torsor_action}
There is a transitive action of $H^1(P_\bullet,T^n_\CC)$ on the set of isomorphism classes of broken toric varieties with polytopal complex $P_\bullet$ for which $\alpha\in H^1(P_\bullet,T^n_\CC)$ acts on the broken toric variety $X$ with gluing data described by $\tau\in H^1(P_\bullet,T^n_\CC)$ by sending $X$ to the broken toric variety with gluing data described by $\alpha\tau$.
\end{lemma}

Since we mostly care about cohomology (which is invariant under this action), we assume that the gluing data is trivial if the torsor is not mentioned.

There is a natural map $f : X \to P_\bullet$ given by gluing together the moment maps of the components of $X$ to their polytopes. In other words, it is the quotient map of $X$ by the compact part of the torus action on each of the components.

All of our cohomology calculations depend on the following statement.

\begin{manualtheorem}{(Theorem 3.4 of \cite{Sun25})}[]
For $X$ an $n$-dimensional broken toric variety and $f$ the map from $X$ to its polytopal complex $P_\bullet$, there is an isomorphism in the bounded derived category $\mathcal{D}^b_c(P_\bullet)$ of constructible sheaves on $P_\bullet$
$$Rf_*\underline{\QQ}_X \cong \bigoplus_{i=0}^{2n}R^if_*\underline{\QQ}_X[-i].$$
This implies the degeneration at the $E_2$ page of the Leray spectral sequence $E^{pq}_2 = H^p(P_\bullet, R^qf_*\underline{\QQ}_X)$ associated to $f$, as well as the cohomological decomposition
\begin{equation}\label{cohomdecomp}
H^i(X,\underline{\QQ}_X)\cong \bigoplus_{p+q=i}H^q(P_\bullet, R^pf_*\uQQ_X).
\end{equation}
\end{manualtheorem}

Let us introduce some particular hyperplanes that are used (essentially for bookkeeping) in Section \ref{section:bam}.

\begin{definition}\label{Halpha}
Let $X$ be a broken toric variety with polytopal complex $P_\bullet$. Given a $k$-cell $\alpha$ of $P_\bullet$ let $H_\alpha$ denote the $k$-dimensional linear subspace of $\left(\mathbb{R}^n\right)^*$ which contains (a shift of) the moment polytope of the $k$-dimensional smooth projective toric variety defined by $\alpha$. 
\end{definition}

One ought to check that this definition is consistent with the way the polytopal complex is glued together. This follows since moment polytopes are well-defined up to translation. Consider the set of $n$-cells of $P_\bullet$ which contain $\alpha$. The map $f$ restricted to each of the toric components of $X$ corresponding to those cells is a moment map and so we get for each toric component a convex subset of $\left(\mathbb{R}^n\right)^*$. By shifting each of these sets to the origin, we see that the image of the toric variety of $\alpha$ under $f$, as a subset of each, must live in their intersection.

\begin{lemma}\label{quot_descr}
Let $X$ be a broken toric variety with gluing data described by the $T^n$-torsor $\alpha$ and $f$ be the map from $X$ to its polytopal complex $P_\bullet$.  Lemma 2.3 of \cite{Sun25} states that $X$ is a quotient of the total space of $\alpha$ on $P_\bullet$. This quotient can be described fibrewise over $x\in\alpha\in\mathring{\text{Sk}}_k(P_\bullet)$ as the orthogonal projection which sends $T^n = \RR^n/\Gamma$ to \mbox{$H_\alpha^* / \Gamma\cap H^*_\alpha $ }where $\Gamma$ is the standard lattice in $ \RR^n$.
\end{lemma}

\begin{proof}
It is true that when restricted to any toric component of $X$ which contains $f^{-1}(x)$, the fibre of $f$ over $x$ is isomorphic to $H_\alpha^* / \Gamma\cap H_\alpha^* $ where $\Gamma$ is the standard lattice in $\RR^n$. This follows from the Delzant construction of a symplectic toric variety from its associated polytope (\cite{D88}, or see Sections 28 and 29 of \cite{CdS01} for a nice overview) from which one can describe the quotient procedure given in Lemma 1.4 of \cite{DJ91}: An $n$-dimensional (smooth projective) toric variety $X$ with polytope $X$ is obtained from $T^n\times P$ by quotienting the fibre over $\alpha\in\mathring{\text{Sk}}_k(P)$ by all directions orthogonal to $\alpha$.

Since the description of $H_\alpha$ is the same for each such component, the gluing respects it. So the equivalence relation which describes the quotient map takes $T^n = \RR^n/\Gamma$ and identifies all vectors normal to $H_\alpha^*$.
\end{proof}

\begin{corollary}\label{fibre_cohom}
Let $X$ be a broken toric variety, $f$ be the map from $X$ to its polytopal complex $P_\bullet$, and $x\in\alpha\in\mathring{\text{Sk}}_k(P_\bullet)$. Then $H^i(f^{-1}(x),\underline{\QQ}) \cong \bigwedge\nolimits^i H_\alpha$.
\end{corollary}

This concrete description of the fibres of the higher derived pushforwards can be quite useful for calculation. Let us end this section with an explicit cohomology calculation using these ideas.

\begin{lemma}\label{gribby}
For $X$ an $n$-dimensional broken toric variety, one has 
\begin{equation*}
H^j(P_\bullet,R^if_*\underline{\QQ}_X|_{\mathring{\text{Sk}}_i}) = \begin{cases}
\QQ^{|\mathring{\text{Sk}}_i(P_\bullet)|} , &j=i,\\
0, &\text{otherwise}
\end{cases}
\end{equation*}
\end{lemma}

\begin{proof}
Lemma 2.17 of \cite{Sun25} states that $H^j(P_\bullet,R^if_*\underline{\QQ}_X) = 0$ for all $j<i$, using a version of Corollary \ref{fibre_cohom} as a main ingredient. Then the localization sequence in compactly supported cohomology can be used to find $H^i(P_\bullet,R^if_*\uQQ_X|_{\mathring{\text{Sk}}_i})\cong H^i_c(\mathring{\text{Sk}}_i,\uQQ_{\mathring{\text{Sk}}_i})$, and the result then follows by Poincar\'e duality.
\end{proof}

\section{Balloon Animal Maps}\label{section:bam}

In this section we will construct long exact sequences comparing the cohomology of broken toric varieties related to each other via balloon animal maps. 

\begin{definition} A subdivision $P_\bullet^{\cup}$ of an $n$-dimensional polytopal complex $P_\bullet$ is a polytopal complex with the property that every polytope in $P_\bullet$ is a union of polytopes of $P_\bullet^\cup$. A subdivision is called \emph{generic} if every  $k$-face $\alpha$ of $P_\bullet^\cup$ which does not lie in the $k$-skeleton of $P_\bullet$ has the property that $\partial \alpha \subset\mathring{\text{Sk}}_k(P_\bullet)$.
\end{definition}

A generic subdivision $P_\bullet^\cup$ of an $n$-dimensional polytopal complex $P_\bullet$ gives rise to a third polytopal complex, $P^{\cap}_\bullet$, defined as the $(n-1)$-dimensional polytopal complex consisting of all faces among the $P_\bullet^\cup$ which are of dimension $k\leq n-1$ and which do not lie in a $k$-face of $P_\bullet$.

\begin{example}\label{subdivs} Pictured below are three subdivisions of a square $P_\bullet$. The first two are not generic while the third is, for which $P^{\cap}_\bullet$ is also pictured.

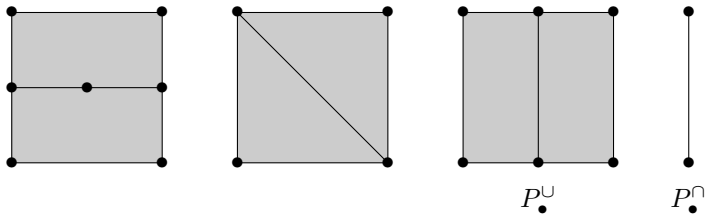
\begin{center}
\begin{tikzpicture}
\fill[gray!40!white] (3,-1) rectangle (5,1);
\draw (3,-1) -- (5,-1) -- (5,1) -- (3,1) -- cycle;
\draw (3,1) -- (5,-1);
\node (b) at (3,-1){$\bullet$};
\node (c) at (5,-1){$\bullet$};
\node (d) at (5,1){$\bullet$};
\node (e) at (3,1){$\bullet$};

\node (b) at (4,-1.5){};

\fill[gray!40!white] (0,-1) rectangle (2,1);
\draw (0,-1) -- (2,-1) -- (2,1) -- (0,1) -- cycle;
\node (b) at (0,-1){$\bullet$};
\node (c) at (2,-1){$\bullet$};
\node (d) at (0,1){$\bullet$};
\node (e) at (2,1){$\bullet$};
\node (e) at (1,0){$\bullet$};
\node (e) at (0,0){$\bullet$};
\node (e) at (2,0){$\bullet$};
\draw (0,0) -- (2,0);
\node (b) at (1,-1.5){};

\fill[gray!40!white] (6,-1) rectangle (8,1);
\draw (6,-1) -- (8,-1) -- (8,1) -- (6,1) -- cycle;
\node (b) at (6,-1){$\bullet$};
\node (c) at (8,-1){$\bullet$};
\node (d) at (8,1){$\bullet$};
\node (e) at (6,1){$\bullet$};
\node (f) at (7,1){$\bullet$};
\node (g) at (7,-1){$\bullet$};
\draw (7,-1)--(7,1);
\node (b) at (7,-1.5){$P^{\cup}_\bullet$};

\node (f) at (9,1){$\bullet$};
\node (g) at (9,-1){$\bullet$};
\draw (9,-1)--(9,1);
\node (b) at (9,-1.5){$P^{\cap}_\bullet$};
\end{tikzpicture}
\end{center}

\end{example}

Now let $X$ be a broken toric variety over $P_\bullet$ with the gluing of its toric components determined by the $T^n$-torsor $\alpha\in H^1(P_\bullet,T^n_\CC)$. Since the polytopal complex $P_\bullet^{\cup}$ is of the same topological type as $P_\bullet$, $\alpha$ also defines a broken toric variety $X^{\cup}$ over $P_\bullet^{\cup}$.

With this setup we can define a map between $X$ and $X^\cup$. The geometric realizations of $P_\bullet$ and $P_\bullet^\cup$ are the same, and $X$ and $X^\cup$ are determined by the same $T^n$ torsor, whose total space we write as $\calT$. So both $X$ and $X^\cap$ are quotients of $\calT$ and the proof of Lemma \ref{quot_descr} describes the equivalence relation defining the quotients. Each equivalence relation depends only on the polytopal complex: more specifically it is fibrewise determined at $x$ by the dimension $k$ for which $x$ lies in the open $k$-skeleton of the polytopal complex. Since $P_\bullet^\cup$ can be viewed as a refinement of $P_\bullet$, the equivalence relation defining $X^\cup$ is a refinement of the one defining $X$. This induces a map from $X$ to $X^\cup$.

\[\begin{tikzcd}
	& {\mathcal{T}} \\
	X && {X^{\cup}}
	\arrow["r"', from=1-2, to=2-1]
	\arrow["{r^{\cup}}", from=1-2, to=2-3]
	\arrow["b", dashed, from=2-1, to=2-3]
\end{tikzcd}\]

\begin{definition}\label{bam_defn}
Given a generic subdivision $P_\bullet^\cup$ of $P_\bullet$, the \emph{balloon animal map} is the continuous map $b:X\to X^\cup$ induced by the refinement of the equivalence relations defining $X$ and $X^\cup$.
\end{definition}

\begin{example}
One example to have in mind is that of the polytopal complexes $P_\bullet = \loopquiver$ and $P_\bullet^{\cup}=\diquiver$.  In this case we have a map of broken toric varieties where $X$ is the torus with one pinched point and $X^{\cup}$ is the torus with two pinched points (so we are ``twisting'' the ``balloon animal'').
\end{example}

\begin{lemma}\label{lem_bam}
Let $b:X\to X^{\cup}$ be a balloon animal map between broken toric varieties induced by a generic subdivision of the polytopal complex of $X$ and $X^{\cap}$ be the broken toric subvariety of $X^\cup$ associated to $P^{\cap}_\bullet$. Then
\begin{enumerate}
\item $R^1b_*\underline{\QQ}_X \cong \underline{\QQ}_{X^{\cap}}.$
\item There is a distinguished triangle in the derived category of sheaves on $X^{\cup}$
$$\underline{\QQ}_{X^{\cup}}\to Rb_*\underline{\QQ}_X\to\underline{\QQ}_{X^\cap}[-1]\xrightarrow[]{+1}$$
\end{enumerate}
\end{lemma}

\begin{proof}
The proof of part $(1)$ is easiest to state using the technology of cellular sheaf cohomology and so is postponed to Section \ref{section:cellular}.

For part ($2$) consider the distinguished triangle from truncation of the pushforward of the constant sheaf on $X$ along the balloon animal map:
\begin{equation*}\tau_{\leq 0}Rb_*\underline{\QQ}_{X}\to Rb_*\underline{\QQ}_{X}\to \tau_{\geq 1}Rb_*\underline{\QQ}_{X}\xrightarrow[]{+1}.
\end{equation*}

By definition of the truncation functor and the fact that $Rb_*\underline{\QQ}_{X}$ is only supported in non-negative degree, $\tau_{\leq 0}Rb_*\underline{\QQ}_{X} \cong R^0Rb_*\underline{\QQ}_{X}\cong \underline{\QQ}_{X^\cup}$. The fibres of $b$ are either $0$- or $1$-dimensional, and so for dimension reasons we have that $R^ib_*\underline{\QQ}_{X}$ is zero for all $i>1$. This implies that $\tau_{\geq 1}Rb_*\underline{\QQ}_{X} = R^1b_*\underline{\QQ}_{X}[-1]$ and so applying part $(1)$ yields the desired statement.
\end{proof}

The assumption of genericity of the subdivision is fundamental in Lemma \ref{lem_bam}, and so for all the results as stated in this section. The main issue is that for balloon animal maps induced by non-generic subdivisions we cannot expect that $R^1b_*\underline{\QQ}_X\cong \underline{\QQ}_{X^\cap}$. Consider the non-generic subdivision of the square into two triangles (pictured in Example \ref{subdivs}).  The stalks of $R^1b_*\underline{\QQ}_X$ are $1$-dimensional over $\mathring{\text{Sk}}_1(P_\bullet^\cap)$ but $0$-dimensional over the $0$-cells.

\begin{theorem}\label{bigballoonpf} Let $b:X\to X^{\cup}$ be a balloon animal map between broken toric varieties induced by a generic subdivision of the polytopal complex of $X$ and $X^{\cap}$ be the broken toric subvariety of $X^\cup$ associated to $P^{\cap}_\bullet$. Then there is a long exact sequence
\begin{equation*}
\dots\to H^{k}(X^\cup,\underline{\QQ}_{X^{\cup}})\to H^k(X,\underline{\QQ}_X)\to H^{k-1}(X^\cap,\underline{\QQ}_{X^\cap})\to\dots.
\end{equation*}
\end{theorem}

\begin{proof}

This follows by applying the hypercohomology functor to the distinguished triangle of Lemma \ref{lem_bam} ($2$) and the fact that $\HH^\bullet(X^\cup,Rb_*\underline{\QQ}_{X}) \cong H^\bullet(X,\underline{\QQ}_{X})$.

\end{proof}

It turns out that the above statement can be refined further after passing to the higher derived pushforwards.

\begin{lemma}\label{bampushforward}
Let $b:X\to X^{\cup}$ be a balloon animal map between broken toric varieties induced by a generic subdivision of the polytopal complex of $X$ and $X^{\cap}$ be the broken toric subvariety of $X^\cup$ associated to $P^{\cap}_\bullet$.  For $i\geq 0$ there exist short exact sequences
\begin{equation*}
0 \to R^if_*\underline{\QQ}_{X^\cup}\to  R^if_*\underline{\QQ}_{X}\to  R^{i-1}f_*\underline{\QQ}_{X^\cap}\to 0.
\end{equation*}
\end{lemma}

\begin{proof}

Taking the derived pushforward to $P^\cup_\bullet$ of the distinguished triangle of Lemma \ref{lem_bam} ($2$), we obtain the distinguished triangle
$$Rf_*\underline{\QQ}_{X^\cup}\to Rf_*Rb_*\underline{\QQ}_{X^\cup} \cong Rf_*\underline{\QQ}_X\to Rf_*\underline{\QQ}_{X^\cap}[-1]\xrightarrow[]{+1},$$
 of which we can take the cohomology objects to obtain the following long exact sequence of sheaves on $P_\bullet$,

\begin{equation}\label{bdryarezero}
\dots\to R^if_*\underline{\QQ}_{X^\cup}\to R^if_*\underline{\QQ}_{X}\to R^{i-1}f_*\underline{\QQ}_{X^\cap}\to \dots
\end{equation}
We want to show that the boundary maps are zero or equivalently that the maps $b_i:R^if_*\underline{\QQ}_{X^\cup}\to R^if_*\underline{\QQ}_{X}$ are injective.  This follows directly from the construction as we will see by observing what is happening on the stalks. 

Let $x\in\alpha\in\mathring{\text{Sk}}_k(P_\bullet)$. The map $b_i$ descends from a balloon animal map and so is the identity away from the `new' cells in the decomposition. That is to say, if $x$ lies within the relative interior of a $k$-face in both $P_\bullet$ and $P_\bullet^\cup$, then $(b_i)_x:\bigwedge\nolimits^i H_\alpha \to \bigwedge\nolimits^i H_\alpha$ (cf. Definition \ref{Halpha}) is the identity.

If $\alpha$ does not lie within a $k$-face of $P_\bullet$, then we are looking at $(b_i)_x:\bigwedge\nolimits^i H_\alpha \to \bigwedge\nolimits^i H_\beta$ where $\beta$ is the $k+1$-face of $P_\bullet$ in which $\alpha$ lies. This map is really a map of cohomology groups (recall Corollary \ref{fibre_cohom}) induced by the map of tori $H^*_\beta/\Gamma\cap H_\beta^* \to H^*_\alpha/\Gamma\cap H^*_\alpha$. This map forms part of a split short exact sequence since $H_\alpha$ and $H_\beta$ are defined so that $H^*_\beta/\Gamma\cap H^*_\beta \cong \left(H^*_\alpha/\Gamma\cap H^*_\alpha\right)\times S^1$. Thus $(b_i)_x$ is injective and the boundary maps of the sequence in Equation \eqref{bdryarezero} are zero.

\end{proof}

\begin{theorem}\label{balloonpf}
Let $b:X\to X^{\cup}$ be a balloon animal map between broken toric varieties induced by a generic subdivision of the polytopal complex of $X$ and $X^{\cap}$ be the broken toric subvariety of $X^\cup$ associated to $P^{\cap}_\bullet$.  For $i\geq 0$ there exists a long exact sequence
\begin{equation*}
\begin{tikzpicture}[descr/.style={fill=white,inner sep=1.5pt}]
        \matrix (m) [
            matrix of math nodes,
            row sep=1em,
            column sep=2.5em,
            text height=1.5ex, text depth=0.25ex
        ]
        {  & &\cdots & H^{k-1}(P^\cap_\bullet, R^{i-1}f_*\underline{\QQ}_{X^\cap}) \\
            &H^{k}(P^\cup_\bullet, R^{i}f_*\underline{\QQ}_{X^\cup})
            &H^{k}(P_\bullet, R^{i}f_*\underline{\QQ}_{X})
            &H^{k}(P^\cap_\bullet, R^{i-1}f_*\underline{\QQ}_{X^\cap}) \\
            & H^{k+1}(P^\cup_\bullet, R^{i}f_*\underline{\QQ}_{X^\cup})& \cdots && \\
        };

        \path[overlay,->, font=\scriptsize,>=latex]
        (m-1-3) edge (m-1-4)
        (m-1-4) edge[out=355,in=175] (m-2-2)
        (m-2-2) edge  (m-2-3)
        (m-2-3) edge  (m-2-4)
        (m-2-4) edge[out=355,in=175] (m-3-2)
        (m-3-2) edge (m-3-3);
\end{tikzpicture}
\end{equation*}
\end{theorem}

\begin{proof}
This is the long exact sequence in cohomology associated to the short exact sequence in Lemma \ref{bampushforward}.
\end{proof}

\section{Using Cellular Sheaves}\label{section:cellular}

One way to approach the problem of calculating the cohomology of cell-compatible sheaves is to view them as cellular sheaves. Let us take a brief look at the definition and their cohomology, following \cite{Ghr14,Cur14}. First, just a bit of notation: for a polytopal complex $P_\bullet$, $\sigma\in\mathring{\text{Sk}}_k(P_\bullet)$, and $\tau\in\mathring{\text{Sk}}_{k+1}(P)$, let us write $\sigma\leq_1 \tau$ if $\sigma\in\mathring{\text{Sk}}_k(\tau)$.

\begin{definition}
A \emph{cellular sheaf} (of vector spaces) $\calF$ on a polytopal complex $P_\bullet$ is
\begin{enumerate}
\item the assignment of a vector space $\calF(\sigma)$ to each cell $\sigma$ of $P_\bullet$
\item a linear map $\rho_{\sigma,\tau}:\calF(\sigma)\to\calF(\tau)$ to every pair of  cells $\sigma\leq_1 \tau$. These \emph{restriction} maps are required to satisfy 
\begin{itemize}
\item $\rho_{\sigma,\sigma} = \text{id}_{\calF(\sigma)}$
\item If $\sigma\leq_1\tau\leq_1\kappa$, then $\rho_{\sigma,\kappa} = \rho_{\sigma,\tau} \circ \rho_{\tau,\kappa}$.
\end{itemize}
\end{enumerate}
\end{definition}

Now we can define the cohomology of a cellular sheaf. For a cellular sheaf $\calF$ on $P_\bullet$, define the $k$-th cochain space of $\calF$ to be 
$$C^k(P_\bullet,\calF) = \bigoplus_{\sigma\in\mathring{\text{Sk}}_k(P_\bullet)}\calF(\sigma).$$
To define coboundary maps, we first need to choose an orientation for each cell in our complex. Then for every pair of cells $\sigma\leq_1\tau$, we get a number $[\sigma:\tau]=\pm 1$, which is equal to $1$ if the orientations of $\sigma$ as a subcomplex of $\tau$ matches the orientation of $\sigma$ and $-1$ if not.  The $k$-th coboundary map $\delta^k: C^k(P_\bullet,\calF)\to C^{k+1}(P_\bullet,\calF)$ is then defined as the sum over $\sigma\in\mathring{\text{Sk}}_k(P_\bullet)$ of the maps $\delta_\sigma^k:\calF(\sigma)\subseteq C^k(P_\bullet,\calF)\to C^{k+1}(P_\bullet,\calF)$ defined by 
$$\delta_\sigma^k(v) = \sum_{\sigma\leq_1 \tau}[\sigma:\tau]\rho_{\sigma,\tau}(v).$$
Then the $k$-th cellular cohomology group of $\calF$ on $P_\bullet$ is 
$$H^k(P_\bullet,\calF) = \frac{\text{ker\,}\delta^k}{\text{im\,}\delta^{k-1}}.$$

This down-to-earth method is useful for calculation when the polytopal complexes involved are simple enough.

\begin{example}\label{3edgecellularcomp}
Let $\calF$ be the cellular sheaf on $T^2$ with cell complex structure as pictured here:

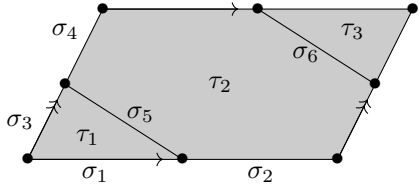
\begin{center}
 \begin{tikzpicture}
 
  \fill[gray!40!white] (0,0) -- (1,2) -- (5.12,2)-- (4.12,0) -- (0,0);
\draw[-] (0,0) -- (1,2) -- (5.12,2)-- (4.12,0) -- (0,0);
\draw (0.5,1) -- (2.06,0);
\draw (3.06,2) --(4.62,1);
\draw[->] (0,0) -- (1.8,0);
\draw[->] (1,2) -- (2.8,2);

\draw[->>] (0,0) -- (0.4,0.8);

\draw[->>] (4.12,0) -- (4.52,0.8);

\node (a) at (0,0){$\bullet$};
\node (b) at (1,2){$\bullet$};
\node (c) at (5.12,2){$\bullet$};
\node (d) at (4.12,0){$\bullet$};
\node (e) at (0.5,1){$\bullet$};
\node (f) at (2.06,0){$\bullet$};
\node (g) at (3.06,2){$\bullet$};
\node (h) at (4.62,1){$\bullet$};

\node (i) at (0.8,0.3){$\tau_1$};
\node (j) at (2.56,1){$\tau_2$};
\node (k) at (4.32,1.7){$\tau_3$};

\node (l) at (-0.1,0.5){$\sigma_3$};
\node (m) at (0.5,1.7){$\sigma_4$};

\node (n) at (0.9,-0.2){$\sigma_1$};
\node (o) at (3.1,-0.2){$\sigma_2$};

\node (p) at (1.5,0.6){$\sigma_5$};
\node (q) at (3.7,1.4){$\sigma_6$};

 \end{tikzpicture}
 \end{center}

Letting $(e_1,e_2)$ be the standard basis vectors, say that the $0$-cells are all assigned the zero vector space, $\sigma_1$ and $\sigma_2$ are assigned $\langle e_1\rangle$, $\sigma_3$ and $\sigma_4$ are assigned $\langle e_2\rangle$, $\sigma_5$ and $\sigma_6$ are assigned $\langle e_1-e_2\rangle$, $\tau_i$ are all assigned $\langle e_1, e_2\rangle$, and the restriction maps are all inclusions. Finally, let us orient $\sigma_1$ and $\sigma_2$ towards the right, $\sigma_3$ and $\sigma_4$ upwards, $\sigma_5$ and $\sigma_6$ pointing up and to the left, and each $2$-cell with the counterclockwise orientation. The cochain complex of this cellular sheaf is 
$$0\to\langle e_1 \rangle^{\oplus 2}\oplus \langle e_2 \rangle^{\oplus 2}\oplus \langle e_2-e_1 \rangle^{\oplus 2}\to \langle e_1,e_2 \rangle^{\oplus 3}\to 0.$$
So the only nontrivial coboundary morphism is $\delta^1$.
 We can calculate, for example, that 
 \begin{align*}
 \delta^k(1,0,0,0,0,0) &=\delta^k_{\sigma_1}(1)\\
 &=\sum_{\sigma_1\leq_1 \tau}[\sigma_1:\tau]\rho_{\sigma_1,\tau}(1)\\
 &=[\sigma_1,\tau_1]\rho_{\sigma_1,\tau_1}(1) + [\sigma_1,\tau_2]\rho_{\sigma_1,\tau_2}(1)
\\
&=(1,0,-1,0,0,0)\in \langle e_1, e_2\rangle^{\oplus 3}
 \end{align*}
 
 Altogether $\delta^k$ can be represented by the matrix
 
 \begin{center}
$ \begin{bmatrix}
    1 & 0  & 0  & 0  & -1 & 0 \\
    0 & 0  & 1  & 0  & 1 & 0\\
   -1 & 1  & 0  & 0  &  1  &-1\\
    0 & 0  & -1 & 1 &   -1& 1\\
    0 & -1 &  0 & 0 &  0& 1 \\
    0 & 0  &  0 & -1   & 0& -1
\end{bmatrix} $
\end{center}
which has rank $4$. Accordingly, the cellular cohomology groups of $\calF$ on $T^2$ are
\begin{equation*}
H^k(T^2,\calF) = \begin{cases}
\QQ^{\oplus 2}, &k=1,2\\
0, &\text{otherwise}
\end{cases}
\end{equation*}

\end{example}

One can also think of cellular sheaves on $P_\bullet$ as usual sheaves on $P_\bullet$ equipped with the \emph{Alexandrov topology} (see for example Section 4.2.1 of \cite{Cur14}). In this context, cellular sheaf cohomology is nothing but the \v{C}ech cohomology of the sheaf.

Cell-compatible sheaves have nice interpretations from this point of view\textemdash{}they can be identified with cellular sheaves where all the restriction maps are inclusions. To be more precise, we can define a cellular sheaf $\calF_{\text{cell}}$ from a cell-compatible sheaf $\calF_{\text{cc}}$ by setting $\calF_{\text{cell}}(\sigma) = \calF_{cc}|_x$ for any $x\in\sigma\in\mathring{\text{Sk}}_k$ and setting all the restriction maps to be inclusions. This construction is independent of $x$ since $\sigma$ is contractible. Moreover, the sheaf cohomology of the cell-compatible sheaf is the same as the sheaf cohomology of its associated sheaf in the Alexandrov topology. The sheaf cohomology of a cell-compatible sheaf then matches with the cellular cohomology of its associated cellular sheaf by Proposition A.15 of \cite{Rus22}.

Using these ideas, we can now provide the postponed proof to part ($1$) of Lemma \ref{lem_bam}.\\

\noindent\emph{Proof of Lemma \ref{lem_bam} ($1$).} First note that $R^1b_*\underline{\QQ}_{X}$ is supported exactly on $Y = f^{-1}(P^\cap_\bullet)$ since for any $x$ away from this locus, the fibre over $x$ is $0$-dimensional and we recall that by proper base change, $(R^1b_*\underline{\QQ}_{X})_x$ is the cohomology group of the fibre over $x$. For $y\in Y$,  the stalk of $R^1b_*\underline{\QQ}_{X}$ is $(R^1f_*\underline{\QQ}_X)_{f(y)} / (R^1f_*\underline{\QQ}_{X^\cup})_{f(y)}$ by the same logic.  In particular all the stalks are $1$-dimensional.  

Denote by $\calF$ and $\calG$ the associated cellular sheaves of $R^1b_*\underline{\QQ}_{X}$ and $\underline{\QQ}_{X^\cap}$ respectively. It is easier to check that we can define an isomorphism of sheaves here than in the classical case. The restriction maps of $\calG$ are the identity by definition and those of $\calF$ take the unique generator of $\calF(\sigma)$ to that of $\calF(\tau)$. This follows from the definition of $R^1b_*\underline{\QQ}_X$ as the sheafification of the presheaf which takes an open set $U$ to $H^1(b^{-1}(U),\underline{\QQ}_X)$. The map $\phi:\calF(\sigma)\to\calG(\sigma)$ which sends generator to generator is then seen to define a cellular sheaf morphism since it commutes with restriction. This morphism is an isomorphism of cellular sheaves since it is on each cell.

Since $\calF$ and $\calG$ are isomorphic, so are their associated sheaves $R^1b_*\underline{\QQ}_{X}$ and $\underline{\QQ}_{X^\cap}$.

\rightline{$\square$}

\section{Hypertoric Hitchin Systems}\label{hths}

In this section we will discuss hypertoric Hitchin systems, leading up to a description of how a particular type of broken toric varieties control their cohomology. Nothing in this section (save for Examples \ref{penult} and \ref{ult}) is novel, but is rather an attempt to present a somewhat contained version of the ideas, following \cite{HP06, MW18, DMS24}. 

\subsection{The Base Space}\label{tate}

Let us begin by defining the space we will be using as our basic building block.  For all $m\in\ZZ$, let $X_m=\{(x_m,y_m):x_m,y_m\in\CC\} \cong \CC^2$ and consider the isomorphism $$f_m:X_m\setminus\{x_m=0\} \to X_{m+1}\setminus\{y_{m+1}=0\}$$
given by
$$(x_m,y_m)\mapsto (x_m^2y_m,x_m^{-1}).$$
With these maps, we can define 
$$\tilde{\frakD} = \left(\bigsqcup_{m\in\ZZ}X_m\right) / \sim,$$ 
where the copies of $\CC^2$ are glued to each other by the maps $f_m$. Put differently, this gluing identifies $(x_m,y_m)$ and $(x_{m+1},y_{m+1})$ if $x_m=y_{m+1}^{-1}$ and $x_my_m=x_{m+1}y_{m+1}$. More generally,$(x_m,y_m)$ and $(x_{m+k},y_{m+k})$ are equivalent if $x_{m+k}=(x_my_m)^kx_m$ and $y_{m+k} = (x_my_m)^{-k}y_m$.  The resulting space is a complex surface which inherits a symplectic form by gluing together the pullbacks $f^*\omega_{m+1}=\omega_m$ of the standard forms $\omega_m = dx_m\wedge dy_m$ on $X_m$. 

Importantly, there is also a proper morphism $\tilde{h}:\tilde{\frakD}\to\CC$ given by $\overline{(x_m,y_m)}\mapsto x_my_m$, which is well-defined by our second description of the gluing. The generic fibre of $h$ is $\{\overline{(x_m,y_m)}\in\tilde{\frakD}: x_my_m=c\neq 0\}\cong \CC^*$ while the fibre over $0$ is an infinite chain of copies of $\PP^1$.  This can be seen by first considering the preimage of $0$ in $\bigsqcup_{m\in\ZZ}X_m$: The preimage of $0$ in each copy of $X_m$ is the union of $\CC=\{(x_m,0)\}$ and $\CC=\{(0,y_m)\}$ intersecting at the origin.  Our equivalence then glues $\{(x_m,0)\}$ to $\{(0,y_{m+1})\}$ via $x_m=y_{m+1}^{-1}$ for all $m\in\ZZ$, which yields an infinite number of copies of $\PP^1$, each one glued to the next at a single point.

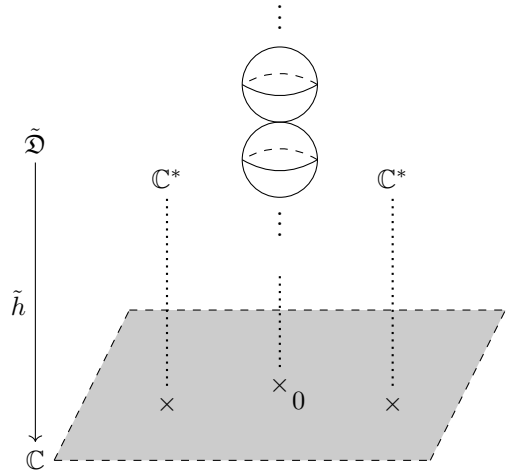
\begin{center}
 \begin{tikzpicture}
 
 \fill[gray!40!white] (0,0) -- (1,2) -- (6,2)-- (5,0) -- (1,0);
  \node (a) at (-0.25,0) {$\CC$};
\draw[dashed] (0,0) -- (1,2);
\draw[dashed] (0,0) -- (5,0);
\draw[dashed] (1,2) -- (6,2);
\draw[dashed] (5,0) -- (6,2);
\node (z) at (-0.25,4.25) {$\tilde{\mathfrak{D}}$};
  \draw[->] (z) -- (a) node[pos=0.5, left]{$\tilde{h}$};
 %

   \node (g) at (3.25,0.8) {$0$};
   \node (b) at (3,1) {$\times$};
   \draw[thick, -, dotted] (b)--(3,2.5);
   
   \draw (3,5) circle (0.5);
    \draw (3,4) circle (0.5);
    
    \draw[dashed] (3.5,5) to [bend right] (2.5,5);
  \draw (3.5,5) to [bend left] (2.5,5);
  
  \draw[dashed] (3.5,4) to [bend right] (2.5,4);
  \draw (3.5,4) to [bend left] (2.5,4);
    
      \node (b) at (3,3.25) {$\vdots$};
        \node (b) at (3,6) {$\vdots$};
  %
  
     \node (c) at (1.5, 0.75) {$\times$};
   \draw[thick, -, dotted] (c)--(1.5, 3.5);
   \node (f) at (1.5, 3.75) {$\CC^*$};
  %
  
       \node (d) at (4.5, 0.75) {$\times$};
   \draw[thick, -, dotted] (d)--(4.5, 3.5);
     \node (e) at (4.5, 3.75) {$\CC^*$};

 \end{tikzpicture}
 \end{center}

There is a also a natural action of $\ZZ$ on $\tilde{\frakD}$ generated by the shift $X_m\to X_{m+1}$, $1:\overline{(x_m,y_m)}\mapsto \overline{(x_{m},y_{m})}$.  From this description, a point of $\tilde\frakD$ is fixed by $k$ if $\overline{(x_m,y_m)}= \overline{((x_my_m)^{k}x_m,(x_my_m)^{-k}y_m)}$, i.e. if it lies over a $k$-th root of unity.  In particular, this tells us that the action of $\ZZ$ on $\tilde{\frakD}$ is free as long as we restrict to points which are mapped by $\tilde{h}$ to the unit disc $\DD\subset\CC$. Accordingly, we define
$$\frakD = \tilde{h}^{-1}(\DD)/\ZZ.$$

It is from this space (the \emph{Tate curve}) that all other hypertoric Hitchin systems are defined.  Clearly $\tilde{h}$ descends to a proper map $h:\frakD\to\DD$. The generic fibre of $h$ is $\CC^*/\ZZ$, an elliptic curve, while the fibre over $0$ is the nodal elliptic curve, as illustrated by how $\ZZ$ acts on the chain of $\PP^1$'s: it maps one to the next, and so in the quotient they are all identified.

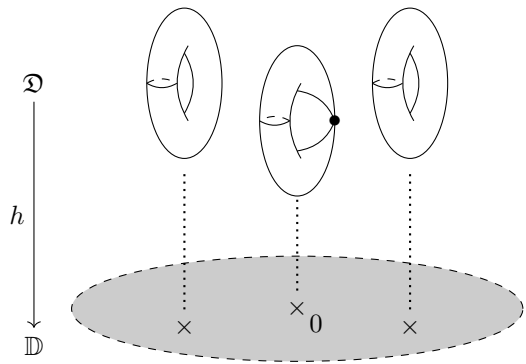
\begin{center}
 \begin{tikzpicture}
  \fill[gray!40!white] (3,1) ellipse (3cm and 0.7cm);
   \draw[dashed] (3,1) ellipse (3cm and 0.7cm);
  \node (a) at (-0.5,0.5) {$\DD$};
  \node (z) at (-0.5,4) {$\mathfrak{D}$};
  \draw[->] (z) -- (a) node[pos=0.5, left]{$h$};
 %

   \node (g) at (3.25,0.8) {$0$};
   \node (b) at (3,1) {$\times$};
   \draw[thick, -, dotted] (b)--(3,2.5);
    \draw (3,3.5) ellipse (0.5cm and 1cm);
 \draw (3.05,3) to [bend left] (3.05,4);
  \draw (3,3.1) to [bend right] (3.5,3.5);
   \draw (3,3.9) to [bend left] (3.5,3.5);
      \node (h) at (3.5,3.5) {$\bullet$};
  \draw[dashed] (2.9,3.5) to [bend right] (2.5,3.5);
  \draw (2.9,3.5) to [bend left] (2.5,3.5);
  %
  %
  
    \node (c) at (1.5, 0.75) {$\times$};
   \draw[thick, -, dotted] (c)--(1.5, 2.8);
      \draw (1.5,4) ellipse (0.5cm and 1cm);
    \draw (1.5,3.6) to [bend right] (1.5,4.4);
  \draw (1.55,3.5) to [bend left] (1.55,4.5);
    \draw (1,4) to [bend right] (1.4,4);
  \draw[dashed] (1,4) to [bend left] (1.4,4);
  %
  
       \node (d) at (4.5, 0.75) {$\times$};
   \draw[thick, -, dotted] (d)--(4.5, 2.8);
      \draw (4.5,4) ellipse (0.5cm and 1cm);
      \draw (4.5,3.6) to [bend right] (4.5,4.4);
  \draw (4.55,3.5) to [bend left] (4.55,4.5);
   \draw (4,4) to [bend right] (4.4,4);
  \draw[dashed] (4,4) to [bend left] (4.4,4);

 \end{tikzpicture}
 \end{center}

\subsection{A Quasi-Hyperk\"ahler Reduction}

As in the additive case, there is an action of $U(1)$ on $\tilde{\frakD}$ given by $\zeta \overline{(x,y)} = \overline{(\zeta x, \zeta^{-1}y)}$ with moment map $\tilde{\mu}:\tilde{\frakD}\to\RR$.  This action descends to an action of $U(1)$ on $\frakD$ with a group-valued moment map $\mu^{\frakD}_{U(1)}:\frakD\to\RR/\ZZ$.  

So, given a short exact sequence of tori $$1\to K\to T^n\to D\to 1,$$ the coordinate-wise action of $T^n$ on $\frakD^n$ induces an action of $K$ with a quasi-hyperk\"ahler moment map 
$$\mu:\frakD^n\to K^*\times \frakk^*_\CC.$$
That is, the action is Hamiltonian for two of the K\"ahler forms but only quasi-Hamiltonian with respect to the third.

There is a general theory of group-valued moment maps (see for example \cite{AMM98}), but it is enough for us to simply define a quasi-hyperk\"ahler reduction in effectively the same way as the hyperk\"ahler case. That is, the \emph{multiplicative hypertoric Dolbeault space} or \emph{hypertoric Hitchin system} associated to $K\hookrightarrow T^n$ and $\alpha\in K^*$ is defined to be 
$$\frakD_\alpha(K\hookrightarrow T^n) = \frakD^n \ssslash_{(\alpha,0)}K = \mu^{-1}(\alpha,0)/K.$$

\subsection{Cographical Hypertoric Hitchin Systems}
The hypertoric Hitchin systems we are interested in correspond to embeddings $K\hookrightarrow T^n$ which arise from graphs.  Such hypertoric Hitchin systems are called \emph{cographical}, and they include many interesting examples.  In particular, they are exactly the spaces that model neighbourhoods of certain fibres of the moduli space of Higgs bundles.

Let $\Gamma$ be a graph with no bridges\footnote{A \emph{bridge} is an edge in a connected graph $\Gamma$ whose removal disconnects $\Gamma$. Bridges in a graph $\Gamma$ do not contribute any information involved in the construction of $\frakD(\Gamma)$ and so can be safely ignored.  To be precise, if $e$ is a bridge of $\Gamma$, then $\frakD(\Gamma) = \frakD(\Gamma/e)$.} and choose an orientation.  This is required for the construction of the (co)homology groups of the graph, but all resulting spaces are independent of the orientation.

From such an oriented graph and a group $G$, let $V(\Gamma)$ and $E(\Gamma)$ denote the vertex and edge sets, respectively. Viewing these sets as the $0$- and $1$-simplices of the CW complex $\Gamma$, we obtain chain complexes
$$0\to C_1(\Gamma,G)=G^{E(\Gamma)}\xrightarrow{d_\Gamma} C_0(\Gamma,G)=G^{V(\Gamma)}\to 0$$
and 
$$0\to C^0(\Gamma,G)=G^{V(\Gamma)}\xrightarrow{d^*_\Gamma} C^1(\Gamma,G)=G^{E(\Gamma)}\to 0.$$

We sometimes denote the dimension of the first homology group of $\Gamma$ as $b_1(\Gamma)$, standing for the \emph{first Betti number} of $\Gamma$.

Taking (co)homology with $U(1)$ coefficients, we find dual short exact sequences of tori
$$1\to H_1(\Gamma,U(1))\to C_1(\Gamma,U(1))\to\text{im}d_\Gamma\to 1$$
and
$$1\to\text{im}d_\Gamma^*\to C^1(\Gamma,U(1))\to H^1(\Gamma,U(1))\to 1.$$

Using this second sequence (setting $K=\text{im}d_\Gamma^*$, etc.), we can define the hypertoric Hitchin system associated to the graph $\Gamma$ and $\alpha\in \text{im}d_\Gamma$ to be $$\frakD(\Gamma) = \frakD^{E(\Gamma)}\ssslash_{(\alpha,0)}\text{im}d_\Gamma^*.$$

Given $\Gamma$ and $A\in U(1)^{E(\Gamma)}$, set $S(A)$ to be the set of edges for which $A_e$ is zero. We say that $\alpha$ is \emph{generic} if for all $A\in d^{-1}_\Gamma(\alpha)$, the graph $\Gamma \setminus S(A)$ is connected. For two stability parameters generic in this sense, the spaces are diffeomorphic, which is why it has been dropped from the notation. From here on, hypertoric Hitchin system" will be taken to mean a cographical hypertoric Hitchin system with $\alpha$ chosen generically.

The moment map $\mu:\frakD^{E(\Gamma)}\to K^*\times \frakk^*_\CC \subset U(1)^{V(\Gamma)}\times \CC^{V(\Gamma)}$ can be described as a composition of the moment map on $\frakD$ with $d_\Gamma$:
$$\frakD^{E(\Gamma)}\xrightarrow{(\mu^{\frakD}_{U(1)}\times h)^{E(\Gamma)}}(U(1)\times \CC)^{E(\Gamma)}  = C_1(\Gamma,U(1)\times \CC)\xrightarrow{d_\Gamma} C_0(\Gamma,U(1)\times \CC) = (U(1)\times \CC)^{V(\Gamma)}.$$

From this, one can describe the inverse image of $(\alpha,0)$: a point $$(\overline{(x^1_m,y^1_m)}, \dots , \overline{(x^{|E(\Gamma)|}_m,y^{|E(\Gamma)|}_m)})\in\frakD^{E(\Gamma)}$$ lies in $\mu^{-1}(\alpha,0)$ if
$$\sum_{\text{edges }e\text{ entering }v}x^e_my^e_m - \sum_{{\text{edges }e\text{ exiting }v}}x^e_my^e_m =0$$
and
$$\prod_{{\text{edges }e\text{ entering }v}}\mu^{\frakD}_{U(1)}(x^e_m,y^e_m)\prod_{{\text{edges }e\text{ exiting }v}}\mu^{\frakD}_{U(1)}(x^e_m,y^e_m)^{-1} = \eta_v$$
for all vertices $v\in V(\Gamma)$.

\subsection{Properties}\label{props}

$\frakD(\Gamma)$ is a smooth manifold of complex dimension $2b_1(\Gamma)$. It comes equipped with a proper morphism $$h_\Gamma:\frakD(\Gamma)\to \DD^{b_1(\Gamma)},$$ the hypertoric analogue of the Hitchin map, which endows it with the structure of an integrable system. As with the usual Hitchin system, there is a deformation retract of $\frakD(\Gamma)$ to its central fibre.  In particular this implies that the inclusion $h_\Gamma^{-1}(0) \hookrightarrow \frakD(\Gamma)$ induces an isomorphism on cohomology. We denote the constant sheaf with rational coefficients on $\frakD(\Gamma)$ by $\underline{\QQ}_{\frakD(\Gamma)}$ and on $h^{-1}(0)_\Gamma$ by $\underline{\QQ}_\Gamma$.

As a seeming aside, let us describe how a graph $\Gamma$ gives rise to a periodic arrangement of hyperplanes in a vector space. There is a natural pairing
$$H^1(\Gamma,\RR)\times H_1(\Gamma,\RR)\to \RR$$
so the class $[e]\in H^1(\Gamma,\RR)$ of an edge $ e \in E(\Gamma)$ defines a linear form 
$$[e]:H_1(\Gamma,\RR)\to \RR.$$
The kernel of this map is a hyperplane $H_e$ in $H_1(\Gamma,\RR)$, so we can find an arrangement of hyperplanes by considering all $e\in E(\Gamma)$. We can pass to a periodic arrangement of hyperplanes in $H_1(\Gamma,\RR)$ by adding in the hyperplanes associated to $\langle e,\eta \rangle=n$ for $n\in\ZZ$, then to an arrangement of hyperplanes in $H_1(\Gamma,\RR) / H_1(\Gamma,\ZZ)$ by simply taking the quotient.

We would like this arrangement to have the property that any $k$ hyperplanes intersect in dimension $k$. This can be ensured by defining the hyperplanes not by $\langle e,\eta \rangle=0$ but by a ``shifted" condition depending on a generic $\alpha\in \text{im}d_\Gamma$. Let $\alpha^*\in K \subset C^0(\Gamma,U(1))$ be the dual of $\alpha$ and let $\tilde{\alpha} = (\tilde{\alpha}_{e_1},\dots,\tilde{\alpha}_{e_{E(\Gamma)}})$ be the image of $\alpha^*$ under $d_\Gamma^*$. Finally, we can define the hyperplane $h_e$ in $H_1(\Gamma,\RR)$ associated to $e$ by $\langle e,\eta \rangle=\tilde{\alpha}_e$.  In other words, we have associated to each edge of $\Gamma$ an element of $U(1)$ and then taken each central hyperplane $H_e$, shifted it by that element, and called the resulting hyperplane $h_e$. See the next section for some examples.

The relevance of this construction is evident from the following theorem of McBreen--Webster.

\begin{theorem}(Proposition 4.9 of \cite{MW18})
The central fibre $h^{-1}_\Gamma(0)$ of $\frakD(\Gamma)$ is a broken toric variety whose polytopal complex has the topological type of real torus of dimension $b_1(\Gamma)$.  Further, the polytopal decomposition on the torus which gives the polytopal complex of $h^{-1}_\Gamma(0)$ is described by a periodic hyperplane arrangement of $\Gamma$.
\end{theorem}

The idea of the proof is to start with the observation that $\tilde{h}_\Gamma^{-1}(0)$ is a symplectic reduction of $\tilde{h}^{-1}(0)^{|E(\Gamma)|}$, which is an infinite grid of copies of $\prod_{e\in E(\Gamma)}\PP^1$.  The reduction is then a quotient of each component of $\tilde{h}^{-1}(0)$ which lies in the fibre of the moment map over $\alpha$.  The quotient of a toric variety by a torus action is toric, so the result is a periodic, non-compact broken toric variety whose polytopal complex is also periodic, with the topological type of $\RR^{b_1(\Gamma)}$.   Taking the quotient by $\ZZ^{b_1(\Gamma)}$ then recovers $h^{-1}_\Gamma(0)$.

As long as the shifts $\tilde{\alpha}_e$ defining the hyperplane arrangement are rational, the complex can be scaled so that the polytopes appearing are lattice polytopes.The hyperplane arrangement thus makes sense as the polytopal complex of a broken toric variety.

\subsection{Examples}\label{exies}

\begin{example}
Let $\Gamma = \loopquiver$. According to the construction of the previous section, we define 
$$\frakD(\Gamma) = \frakD \ssslash_{(\alpha,0)}\text{im}d_\Gamma^*.$$
In this case, $d^*_\Gamma$ is the zero map and every $(x,y)\in\frakD$ lies in $\mu^{-1}(\alpha,0)$. Thus $\frakD(\Gamma) = \frakD$.

We can also verify that the central fibre of $\frakD$ is what we expect it to be. Recall that by construction, $h^{-1}_\Gamma(0)$ should be a pinched torus. On the other hand, the periodic hyperplane arrangement associated to $\Gamma = \loopquiver$ is $\ZZ\subset \RR$, the quotient of which by $H^1(\Gamma, \ZZ) = \ZZ$ gives the hyperplane arrangement of a single point on $S^1$.  The fibre of the map $f:h^{-1}_\Gamma(0)\to S^1$ is a copy of $S^1$ away from this hyperplane and a point over the hyperplane.  
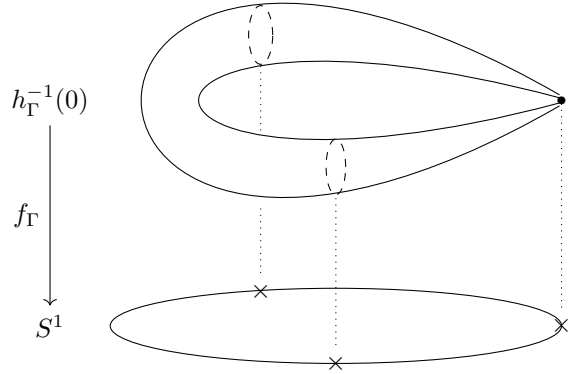
\vspace{-25mm}
\begin{center}
 \begin{tikzpicture}
  \draw (0,0) ellipse (3cm and 0.5cm);
%

\node(a) at (3.1,3) {};
\node (b) at (-3,3) {};
\draw[] (a) to [out=165,in=195, looseness=250] (a);
\draw [] (a) to [out=150,in=210, looseness=150](a);

   \draw[dashed] (-1,3.87) ellipse (0.16cm and 0.39cm);
     \node (b) at (-1,0.46) {$\times$};
   \draw[dotted] (b)--(-1,1.6);
     \draw[dotted] (-1,2.6)--(-1,3.46);
  %
  
   \draw[dashed] (0,2.115) ellipse (0.13cm and 0.37cm);
     \node (c) at (0,-0.5) {$\times$};
   \draw[dotted] (c)--(0,1.7);
  %
  
     \node (d) at (3,0) {$\times$};
   \draw[dotted] (d)--(3,3);
   \node at (3,3) {$\scriptstyle\bullet$};
  \node (x) at (-3.8,3) {$h_\Gamma^{-1}(0)$};
   \node (y) at (-3.8,0) {$S^1$};
  \draw[->] (x) -- (y) node[pos=0.5, left]{$f_\Gamma$};

 \end{tikzpicture}
 \end{center}
 
\end{example}

\begin{example}
Let $\Gamma = \triquiver$.  Then
$$\frakD(\Gamma) = \frakD^3 \ssslash_{(\alpha,0)}\text{im}d_\Gamma^*.$$
This a $4$-complex-dimensional space equipped with a proper map $h_\Gamma:\frakD(\Gamma)\to\DD^2$ whose generic fibre is $T^4$. For this example, let us go carefully through the construction of the periodic hyperplane arrangement of $\Gamma$ describing the structure of the central fibre. We need to choose an orientation for the graph:
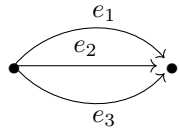
\begin{center}
 \begin{tikzpicture}
\node(a) at (0,-0.04) {$\bullet$};
\node(b) at (2.1,-0.04) {$\bullet$};

\draw[->] (0,0) -- (1.9,0) node[pos=0.5, above]{$e_2$};
\draw[->] (0,0) to [out = 50, in =130] (2,0.1) {};
\draw[->] (0,0) to [out = 310, in = 230] (2,-0.1) {};

\node(c) at (1.2,0.7) {$e_1$};
\node(d) at (1.2,-0.7) {$e_3$};
\end{tikzpicture}
\end{center}

We can then write $H_1(\Gamma,\RR) = \langle [e_1-e_2], [e_2-e_3]\rangle$ and we have linear forms $[e_i]:H_1(\Gamma,\RR)\to \RR$. Given $a,b\in\RR$, these forms look like
$$[e_1](a[e_1-e_2], b[e_2-e_3]) = a$$
$$[e_2](a[e_1-e_2], b[e_2-e_3]) = -a+b$$
$$[e_3](a[e_1-e_2], b[e_2-e_3]) = b$$
and thus the arrangement of central hyperplanes they define in $H_1(\Gamma,\RR)$ is given by the $x$ and $y$ axes, along with the line $-x+y=0$.  However, this arrangement is not stable (there are $3$ hyperplanes meeting at $0$) so we shift the hyperplane $H_{e_2}$ by some $\tilde{\alpha}_{e_2}\in S^1$. We then pass to a periodic arrangement and then to an arrangement on $T^2 = H_1(\Gamma,\RR)/H_1(\Gamma,\ZZ)$. This is the base in the following picture, which also visualizes some of the fibres of the map $f:h^{-1}_\Gamma(0)\to T^2$.

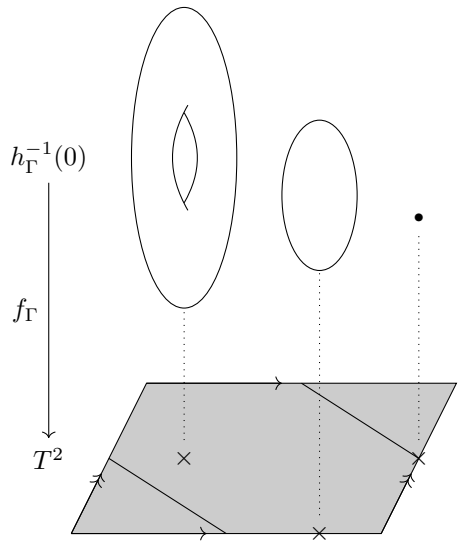
\begin{center}
 \begin{tikzpicture}
 
  \fill[gray!40!white] (0,0) -- (1,2) -- (5.12,2)-- (4.12,0) -- (0,0);
\draw[-] (0,0) -- (1,2) -- (5.12,2)-- (4.12,0) -- (0,0);
\draw (0.5,1) -- (2.06,0);
\draw (3.06,2) --(4.62,1);
\draw[->] (0,0) -- (1.8,0);
\draw[->] (1,2) -- (2.8,2);

\draw[->>] (0,0) -- (0.4,0.8);

\draw[->>] (4.12,0) -- (4.52,0.8);

 %

  \draw (1.5,5) ellipse (0.7cm and 2cm);
  \draw (1.5,4.4) to [bend right] (1.5,5.6);
  \draw (1.55,4.3) to [bend left] (1.55,5.7);
   \node (a) at (1.5,1) {$\times$};
   \draw[dotted] (a)--(1.5,3);
  %
  
   \draw (3.3,4.5) ellipse (0.5cm and 1cm);
     \node (b) at (3.3,0) {$\times$};
   \draw[dotted] (b)--(3.3,3.5);
  %
  
     \node (c) at (4.62,1) {$\times$};
   \draw[dotted] (c)--(4.62,4);
   \node at (4.62,4.2) {$\scriptstyle\bullet$};

   \node (x) at (-0.3,5) {$h_\Gamma^{-1}(0)$};
   \node (y) at (-0.3,1) {$T^2$};
  \draw[->] (x) -- (y) node[pos=0.5, left]{$f_\Gamma$};
 \end{tikzpicture}
 \end{center}

The irreducible components of the broken toric variety $h^{-1}_\Gamma(0)$ are two copies of $\PP^2$ (the toric varieties of the two triangular cells in the arrangement) along with a copy of $\PP^1\times\PP^1$ which is blown up at two points (the toric variety of the hexagon).
\end{example}

\begin{example}\label{penult}
Let $\Gamma = \newquiver$. This example is nearly the same as the previous, but with the edge $e_1$ divided into two edges ($e_1'$ and $e_1''$).  They both define the same central hyperplane, so one must be shifted in order to get a stable arrangement.

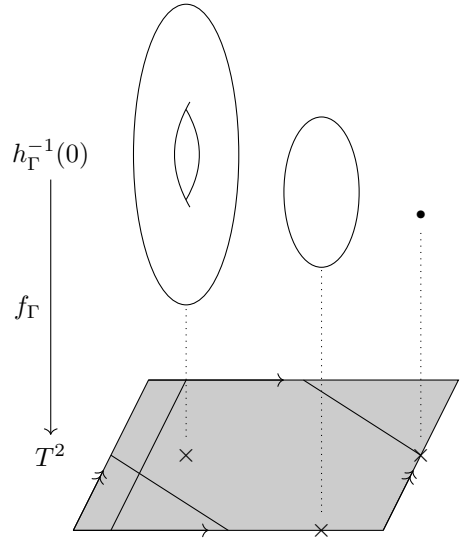
\begin{center}
 \begin{tikzpicture}
 
  \fill[gray!40!white] (0,0) -- (1,2) -- (5.12,2)-- (4.12,0) -- (0,0);
\draw[-] (0,0) -- (1,2) -- (5.12,2)-- (4.12,0) -- (0,0);
\draw (0.5,1) -- (2.06,0);
\draw (0.5,0) -- (1.5,2);
\draw (3.06,2) --(4.62,1);
\draw[->] (0,0) -- (1.8,0);
\draw[->] (1,2) -- (2.8,2);

\draw[->>] (0,0) -- (0.4,0.8);

\draw[->>] (4.12,0) -- (4.52,0.8);

 %

  \draw (1.5,5) ellipse (0.7cm and 2cm);
  \draw (1.5,4.4) to [bend right] (1.5,5.6);
  \draw (1.55,4.3) to [bend left] (1.55,5.7);
   \node (a) at (1.5,1) {$\times$};
   \draw[dotted] (a)--(1.5,3);
  %
  
   \draw (3.3,4.5) ellipse (0.5cm and 1cm);
     \node (b) at (3.3,0) {$\times$};
   \draw[dotted] (b)--(3.3,3.5);
  %
  
     \node (c) at (4.62,1) {$\times$};
   \draw[dotted] (c)--(4.62,4);
   \node at (4.62,4.2) {$\scriptstyle\bullet$};

\node (x) at (-0.3,5) {$h_\Gamma^{-1}(0)$};
   \node (y) at (-0.3,1) {$T^2$};
  \draw[->] (x) -- (y) node[pos=0.5, left]{$f_\Gamma$};
 \end{tikzpicture}
 \end{center}
\end{example}

\begin{example}\label{ult}
Let $\Gamma = \fourquiver$. This gives a higher-dimensional example since $b_1(\Gamma) = 3$, the central hyperplanes of which are given by the planes defined by $x=0$, $y=0$, $z=0$, and $x+y+z=0$.  Pictured below is a cube representing $T^3$ with the faces of the cube and the interior diagonal plane representing the shifted periodic hyperplane arrangement. Analogously to the previous examples, the fibre of $f_\Gamma$ over a point in $\mathring{\text{Sk}}_k$ is $T^k$. Further, it appears that the broken toric variety $h_\Gamma^{-1}(0)$ has four irreducible components (see Lemma \ref{tree}).
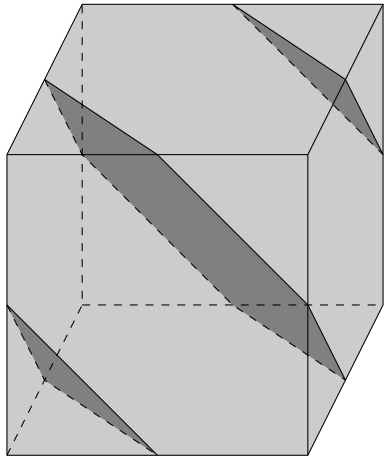
\vspace*{3pt}
\begin{center}
 \begin{tikzpicture}

\fill[gray!40!white] (0,0) -- (0,4)--(1,6)--(5,6)--(5,2)--(4,0)--(0,0);

 \node (a) at (0,0) {};
  \node (b) at (1,2) {};
   \node (c) at (5,2) {};
    \node (d) at (4,0) {};
     \node (e) at (0.5,1) {};
      \node (f) at (2,0) {};
       \node (g) at (3,2) {};
        \node (h) at (4.5,1) {};
 \node (i) at (0,4) {};
  \node (j) at (1,6) {};
   \node (k) at (5,6) {};
    \node (l) at (4,4) {};
     \node (m) at (0.5,5) {};
      \node (n) at (2,4) {};
       \node (o) at (3,6) {};
        \node (p) at (4.5,5) {};
        
 \node (q) at (0,2) {};
  \node (r) at (1,4) {};
   \node (s) at (5,4) {};
    \node (t) at (4,2) {};
    
\fill[gray] (0,2)--(0.5,1)--(2,0)--(0,2);
\fill[gray] (3,6)--(5,4)--(4.5,5)--(3,6);
\fill[gray] (0.5,5)--(1,4)--(3,2)--(4.5,1)--(4,2)--(2,4)--(0.5,5);

 \draw[-] (0,0) -- (0,4)--(1,6)--(5,6)--(5,2)--(4,0)--(0,0);
 \draw[-] (0,4) -- (4,4) -- (5,6);
 \draw[-] (4,4) -- (4,0);
  \draw[-] (0,2) -- (2,0);
   \draw[dashed] (0,2) -- (0.5,1)--(2,0);
    \draw[-] (0.5,5) -- (2,4) -- (4,2) -- (4.5,1);
     \draw[-] (3,6) -- (4.5,5) -- (5,4);
      \draw[dashed] (0,0) -- (1,2) -- (1,6);
       \draw[dashed] (1,2) -- (5,2);
        \draw[dashed] (0.5,5) -- (1,4)--(3,2)--(4.5,1);
         \draw[dashed] (3,6) -- (5,4);

 \end{tikzpicture}
 \end{center}
\end{example}

\section{Deletion-Contraction}\label{sectiondelcon}

Deletion-contraction relations are often useful when graphs are concerned, as evidenced for example by the number of spanning trees of a graph or its chromatic polynomial (both of which are really specializations of the Tutte polynomial of a graph, which is universal among multiplicative graph invariants that exhibit deletion-contraction relationships).  Hypertoric Hitchin systems are no exception to this rule. Recall that for a graph $\Gamma$ and a non-loop edge $e$, the \emph{deletion} $\Gamma\setminus e$ is the graph which has the same vertices of $\Gamma$ and all the same edges, sans $e$. The \emph{contraction} $\Gamma/e$ is the graph obtained by replacing the two vertices $\{v_1,v_2\}$ to which $e$ is incident with a single vertex, with any edge that was incident to $v_1$ or $v_2$ now incident to the new vertex (and so an edge other than $e$ incident to both $v_1$ and $v_2$ becomes a loop in the contraction graph). 

\begin{example}\label{simpledelcon}
Let $\Gamma = \newquiver$ and $e$ be one of the edges adjacent to the unique vertex of degree two. Then $\Gamma/e = \triquiver$ and $\Gamma\setminus e = \diquiverplus$. 
\end{example}

\begin{example}
Let $\Gamma$ be a connected graph, $e$ a non-loop edge of $\Gamma$, and $t(\Gamma)$ be the number of spanning trees of $\Gamma$. Then
$$t(\Gamma) = t(\Gamma\setminus e)+t(\Gamma/e).$$
This is a prototypical example of a deletion-contraction relationship.
\end{example}

Now we can express some of our earlier machinery involving broken toric varieties in terms of deletion-contraction. 

\begin{proposition}\label{delconballoon}
Let $\Gamma$ be a graph and $e$ be a non-loop edge of $\Gamma$. Associated to the contraction $\Gamma\to\Gamma/e$ is a balloon animal map (\emph{cf.} Section \ref{section:bam}) of broken toric varieties $h^{-1}_{\frakD(\Gamma/e)}(0)\to h^{-1}_{\frakD(\Gamma)}(0)$ for which the broken toric variety $X^\cap$ associated to the intersection of the polytopes in the subdivision is $h^{-1}_{\frakD(\Gamma\setminus e)}(0)$.
\end{proposition}

\begin{proof}
The passage from $\Gamma/e$ to $\Gamma$ at the level of the periodic hyperplane arrangements describing the varieties amounts to the addition of a single hyperplane to the arrangement. Since the hyperplane arrangement is stable, the subdivision of polytopal complexes that it induces is generic and so one can define the balloon animal map between the two broken toric varieties.

The hyperplane arrangement on the hyperplane corresponding to $e$ is equivalently either the polytopal complex $P_\bullet^\cap$ associated to the intersection of the polytopes in the subdivision or the hyperplane arrangement induced by the inclusion of $H_1(\Gamma\setminus e,\RR)\subset H_1(\Gamma,\RR)$. The latter is precisely the hyperplane arrangement associated to $\Gamma\setminus e$ by definition: For any non-bridge edge $e'\in\Gamma\setminus e$, the kernel of $[e']:H_1(\Gamma\setminus e,\RR)\to\mathbb{R}$ is the intersection of the kernels of $[e]:H_1(\Gamma,\RR)\to\mathbb{R}$ and $[e']:H_1(\Gamma,\RR)\to\mathbb{R}$.

\end{proof}

\begin{example}
Let $\Gamma$ and $e$ be as in Example \ref{simpledelcon}. Pictured below are the polytopal complexes of the central fibres of the hypertoric Hitchin systems $\frakD(\Gamma)$ and $\frakD(\Gamma/e)$. The polytopal complex of the central fibre of $\frakD(\Gamma\setminus e)$ is the circle with two marked points.

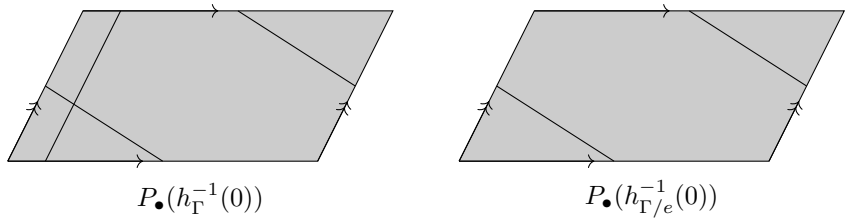
\begin{center}
 \begin{tikzpicture}
  \fill[gray!40!white] (0,0) -- (1,2) -- (5.12,2)-- (4.12,0) -- (0,0);
\draw[-] (0,0) -- (1,2) -- (5.12,2)-- (4.12,0) -- (0,0);
\draw (0.5,1) -- (2.06,0);
\draw (0.5,0) -- (1.5,2);
\draw (3.06,2) --(4.62,1);
\draw[->] (0,0) -- (1.8,0);
\draw[->] (1,2) -- (2.8,2);

\draw[->>] (0,0) -- (0.4,0.8);

\draw[->>] (4.12,0) -- (4.52,0.8);

\node (a) at (2.56,-0.5){$P_\bullet(h^{-1}_{\Gamma}(0))$};

  \fill[gray!40!white] (6,0) -- (7,2) -- (11.12,2)-- (10.12,0) -- (6,0);
\draw[-] (6,0) -- (7,2) -- (11.12,2)-- (10.12,0) -- (6,0);
\draw (6.5,1) -- (8.06,0);
\draw (9.06,2) --(10.62,1);
\draw[->] (6,0) -- (7.8,0);
\draw[->] (7,2) -- (8.8,2);

\draw[->>] (6,0) -- (6.4,0.8);

\draw[->>] (10.12,0) -- (10.52,0.8);

\node (b) at (8.56,-0.5){$P_\bullet(h^{-1}_{\Gamma/e}(0))$};
\end{tikzpicture}
\end{center}
\end{example}

The next theorem is originally due to Dansco--McBreen--Shende. Our investigation of balloon animal maps provides an alternative proof.

\begin{theorem}(Theorem 6.43 of \cite{DMS24})\label{thm_delcon}
Let $\Gamma$ be a graph and $e$ be an edge of $\Gamma$ which is not a loop. Then there is a long exact sequence relating the cohomologies of hypertoric Hitchin systems:
\begin{equation*}
\dots\to H^{k}(\frakD(\Gamma),\underline{\QQ}_{\frakD(\Gamma)})\to H^k(\frakD(\Gamma/e),\underline{\QQ}_{\frakD(\Gamma/e)})\to H^{k-1}(\frakD(\Gamma \setminus e),\underline{\QQ}_{\frakD(\Gamma \setminus e)})\to\dots
\end{equation*}
\end{theorem}

\begin{proof}
Theorem \ref{bigballoonpf} and Proposition \ref{delconballoon} combine to yield the result at the level of central fibres $h^{-1}(0)$, and the inclusion of these fibres into their respective ambient hypertoric Hitchin systems induces an isomorphism on cohomology.
\end{proof}

\begin{theorem}\label{dcpf}
Let $\Gamma$ be a graph and $e$ be an edge of $\Gamma$ which is not a loop. Then there is a long exact sequence
\begin{equation*}
\begin{tikzpicture}[descr/.style={fill=white,inner sep=1.5pt}]
        \matrix (m) [
            matrix of math nodes,
            row sep=1em,
            column sep=1.5em,
            text height=1.5ex, text depth=0.25ex
        ]
        {  & &\cdots & H^{k-1}(T^{n-1}, R^{i-1}f_*\underline{\QQ}_{\frakD(\Gamma\setminus e)}) \\
            &H^{k}(T^n, R^{i}f_*\underline{\QQ}_{\frakD(\Gamma)})
            &H^{k}(T^n, R^{i}f_*\underline{\QQ}_{\frakD(\Gamma/ e)})
            &H^{k}(T^{n-1}, R^{i-1}f_*\underline{\QQ}_{\frakD(\Gamma\setminus e)}) \\
            & H^{k+1}(T^n, R^{i}f_*\underline{\QQ}_{\frakD(\Gamma)})& \cdots && \\
        };

        \path[overlay,->, font=\scriptsize,>=latex]
        (m-1-3) edge (m-1-4)
        (m-1-4) edge[out=355,in=175] (m-2-2)
        (m-2-2) edge  (m-2-3)
        (m-2-3) edge  (m-2-4)
        (m-2-4) edge[out=355,in=175] (m-3-2)
        (m-3-2) edge (m-3-3);
\end{tikzpicture}
\end{equation*}
\end{theorem}

\begin{proof}
This is a special case of Theorem \ref{balloonpf} in the case that the balloon animal maps are between special fibres of hypertoric Hitchin systems and arising from a deletion-contraction relation.
\end{proof}

One more useful tool for computing cohomologies of hypertoric Hitchin systems is the following, which can be thought of as a further refinement of Theorem \ref{dcpf}.

\begin{theorem}\label{smalldelcon}
Let $\Gamma$ be a graph with $b_1(\Gamma) =n \geq 2$ and $e$ an edge of $\Gamma$ which is incident to a vertex of degree $2$. Then 
\begin{equation*}
H^k(T^n, R^if_*\underline{\QQ}_{\Gamma}) \cong H^k(T^n, R^if_*\underline{\QQ}_{\Gamma / e}) \oplus H^{k-1}(T^n, R^{i-1}f_*\underline{\QQ}_{\Gamma\setminus e}).
\end{equation*}
\end{theorem}

\begin{proof}
Consider the long exact sequence of Theorem \ref{dcpf}:
\begin{equation*}
\begin{tikzpicture}[descr/.style={fill=white,inner sep=1.5pt}]
        \matrix (m) [
            matrix of math nodes,
            row sep=1em,
            column sep=2.5em,
            text height=1.5ex, text depth=0.25ex
        ]
        {  & &\cdots & H^{k-1}(T^{n-1}, R^{i-1}f_*\underline{\QQ}_{\Gamma\setminus e}) \\
            &H^{k}(T^n, R^{i}f_*\underline{\QQ}_{\Gamma})
            &H^{k}(T^n, R^{i}f_*\underline{\QQ}_{\Gamma/e})
            &H^{k}(T^{n-1}, R^{i-1}f_*\underline{\QQ}_{\Gamma\setminus e}) \\
            & H^{k+1}(T^n, R^{i}f_*\underline{\QQ}_{\Gamma})& \cdots && \\
        };

        \path[overlay,->, font=\scriptsize,>=latex]
        (m-1-3) edge (m-1-4)
        (m-1-4) edge[out=355,in=175] (m-2-2)
        (m-2-2) edge node[descr,yshift=1.18ex] {$b_k$} (m-2-3)
        (m-2-3) edge node[descr,yshift=1ex] {$\eta_k$} (m-2-4)
        (m-2-4) edge[out=355,in=175] (m-3-2)
        (m-3-2) edge (m-3-3);
\end{tikzpicture}
\end{equation*}

The statement of the theorem is that $\eta_k=0$ for all $k$, which is in turn equivalent to the statement that the map on cohomology $b_k$ induced by the balloon animal map has a section. 

To construct such a section, consider the periodic hyperplane arrangements $A_\Gamma$ and $A_{\Gamma/e}$ associated to $\Gamma$ and $\Gamma/e$ respectively, which describe the cell-compatible sheaves $R^if_*\underline{\QQ}_{\Gamma}$ and $R^if_*\underline{\QQ}_{\Gamma/e}$. The only difference between these arrangements is the presence in $A_\Gamma$ of the shifted hyperplane $h_{e}$ which is parallel to $h_{e'}$ where $e'$ is the other edge incident to $v$. These two hyperplanes are parallel since $e$ and $e'$ lie on all the same cycles in $\Gamma$ so they define the same central hyperplane. They are shifted by different amounts $\tilde{\alpha}_{e},\tilde{\alpha}_{e'}$ to ensure genericity of the arrangement. Taking the difference of their shifts $\tilde{\alpha}_{e}-\tilde{\alpha}_{e'}$ to zero, one essentially forgets the hyperplane $h_{e}$ and so defines a map of cell-compatible sheaves $s:R^if_*\underline{\QQ}_{\Gamma}\to R^if_*\underline{\QQ}_{\Gamma/e}$. Cell-compatibility is preserved by this operation of changing the shifts since the stalks over $h_e$ and $h_{e'}$ are the same and the stalks over all the $n$-cells surrounding them also all match. The map $s$ induces a map in cohomology $H^k(R^if_*\underline{\QQ}_{\Gamma/e})\to H^k(R^if_*\underline{\QQ}_{\Gamma})$ which is a section of $b_k$.
\end{proof}

\section{Cohomology Calculations}\label{cohomo}

In this section we describe a potpourri of results regarding actual calculations of the cohomology of hypertoric Hitchin systems. We start with a statement that allows us to ``break up" the central fibre $\frakD(\Gamma)$ for graphs which exhibit a certain kind of decomposability.

\begin{definition} Let $\Gamma$ be a connected graph. The \emph{induced subgraph} of a subset of edges $E'\subset E(\Gamma)$ is the subgraph of $\Gamma$ containing only the edges of $E'$ and all vertices which those edges are incident to. A vertex $v$ of $\Gamma$ is a \emph{disconnecting vertex} if the set of edges $E(\Gamma)$ of $\Gamma$ can be partitioned into two subsets $E_1,E_2$ such that the intersection of the induced subgraphs of $E_1$ and $E_2$ is $\{v\}$.
\end{definition}

A disconnecting vertex is slightly different than a \emph{cut vertex}, which is usually defined as a vertex $v$ in a graph $\Gamma$ such that if $v$ and all of its incident vertices were removed, the resulting graph would be disconnected. These definitions can differ when the graph has loops. For example, the graph consisting of two loops at a single vertex has a disconnecting vertex but no cut vertices.

\begin{theorem}\label{discondecomp}
Let $v$ be a disconnecting vertex of a connected graph $\Gamma$ with $b_1(\Gamma)=n$ so that $\Gamma$ can be written as the union of a finite number of connected graphs $\Gamma_1,\ldots,\Gamma_m$ with $\bigcap_{i=1}^m\Gamma_i = \{v\}$. Then there exists $T^n$-torsor $\alpha$ on $P_\bullet(h_\Gamma^{-1}(0))$ such that
$$\alpha\cdot h_\Gamma^{-1}(0) \cong \prod_{i=1}^m h_{\Gamma_i}^{-1}(0).$$
\end{theorem}

\begin{proof}
We will first show that the polytopal complexes of $h_\Gamma^{-1}(0)$ and of $\prod_{i=1}^m h_{\Gamma_i}^{-1}(0)$ are identical. This statement follows from the description of the hyperplane arrangements of $\Gamma_i$ which describe the polytopal complexes of $h_{\Gamma_i}^{-1}(0)$. Each $\Gamma_i$ defines an arrangement $A_i$ on $\RR^{b_1(\Gamma_i)}$, and we claim that the arrangement $\prod_{i=1}^m A_i$ on $\prod_{i=1}^m \RR^{b_1(\Gamma_i)}$ is the same as the arrangement $A$ on $\RR^{b_1(\Gamma)}$ defined by $\Gamma$. 

This is true, since given an edge $e$ in $\Gamma_i$ with associated central hyperplane $H'_e$, the morphism $[e]:H(\Gamma_j,\RR)\to \RR$ is zero for any $j\neq i$. That is to say, the central hyperplane $H_e$ in $\RR^{b_1(\Gamma)}$ can be written as $H'_e \times \prod_{k=1, k\neq i}^m \RR^{b_1(\Gamma_i)}$.

Now, by Lemma \ref{torsor_action} there is some $\alpha\in H^1(P_\bullet(h_\Gamma^{-1}(0)),T^n_\CC)$ with the desired property.
\end{proof}

By the Künneth formula we get as a corollary that the cohomology of $\frakD(\Gamma)$ can be decomposed in a similar way.

\begin{corollary}\label{disconcorr}
Let $v$, $\Gamma$, and $\Gamma_i$ be as in Theorem \ref{discondecomp}.  Then the Poincar\'e polynomial of $\frakD(\Gamma)$ is the product of the Poincar\'e polynomials of $\frakD(\Gamma_i)$, or equivalently
$$H^k(\frakD(\Gamma),\underline{\QQ}_{\frakD(\Gamma)}) \cong \bigoplus_{k_1+\dots+k_m=k}\left(\bigotimes_{i=1}^m H^{k_i}(\frakD(\Gamma_i),\underline{\QQ}_{\frakD(\Gamma_i)})\right)$$
\end{corollary}

Next up, we calculate the top-degree cohomology of $\frakD(\Gamma)$ for any graph. Along the way we will use the following definition and lemma.

\begin{definition}
Let $\Theta_k$ denote the graph with two vertices and $k$ edges between them.
\end{definition}

\begin{lemma}\label{reduceme}
Any connected graph $\Gamma$ with no disconnecting vertices can be reduced to $\Theta_{b_1(\Gamma)+1}$ via contraction.
\end{lemma}

\begin{proof}
A connected graph $\Gamma$ with no disconnecting vertices and $b_1(\Gamma)=1$ is simply a cycle graph, which can be transformed to $\Theta_2$ via repeated contraction operations. 

Now assume that the statement holds for any connected graph with no disconnecting vertices and first Betti number $k.$ To show that it holds when $\Gamma$ is a connected graph with no disconnecting vertices and $b_1(\Gamma)=k+1$, choose a subgraph $\Gamma'$ of $\Gamma$ with first Betti number $k$ and let $\Gamma''$ denote the subgraph induced by all edges that not in this chosen subgraph. Apply to $\Gamma$ the contractions which lead to the subgraph $\Gamma'$ being reduced to $\Theta_{k+1}$. Since contractions cannot introduce disconnecting vertices to a graph, $\Gamma''\cap \Theta_{k+1} = V(\Theta_{k+1}).$ This implies that $\Gamma''$ can have no cycles, or else $b_1(\Gamma)$ would be greater than $k+1$. Hence $\Gamma''$ can be transformed to $\Theta_1$ via contractions. Apply these contractions to $\Gamma''$ as a subgraph of $\Gamma$ to finally yield $\Theta_{k+2}$.
\end{proof}

\begin{lemma}\label{tree}
For $\Gamma$ a graph with $b_1(\Gamma)=n$ we have
$$h^{2n}(\frakD(\Gamma),\underline{\QQ}_{\frakD(\Gamma)}) = h^n(T^n,R^nf_*\underline{\QQ}_\Gamma) = |\mathring{\text{Sk}}_n(P_\bullet(h_\Gamma^{-1}(0)))| = t(\Gamma)$$ where $t(\Gamma)$ is the number of spanning trees of $\Gamma$.
\end{lemma}

\begin{proof}
The first equality follows from the cohomological decomposition statement in Equation \ref{cohomdecomp} along with the vanishing given in Lemma 3.2 of \cite{Sun25}. For the second equality, use the localization sequence in compactly supported cohomology to find $H^n(T^n,R^nf_*\uQQ_\Gamma)\cong H^n_c(\mathring{\text{Sk}}_n,\uQQ_{\mathring{\text{Sk}}_n})$ followed by Poincar\'e duality.  The nontrivial part of this statement is that the number of $n$-cells in $P_\bullet(h_\Gamma^{-1}(0))$ is equal to the number of spanning trees of $\Gamma$.  

We will ignore the case of graphs with a disconnecting vertex since the central fibres of the associated hypertoric Hitchin systems can be decomposed by Theorem \ref{discondecomp}. The proof proceeds by induction thanks to Lemma \ref{reduceme}. The statement is true for $\Theta_k$, whose associated (unstable) periodic hyperplane arrangements consist of the sides of the unit cube $[0,1]^{k-1}$ and the hyperplane given by $x_1+\ldots+x_{k-1} =0$, the periodic version of which passes through $[0,1]^{k-1}$ $k-1$ times. To get a stable arrangement we shift the final hyperplane slightly, yielding $k$-many $(k-1)$-cells, which agrees with $t(\Theta_k)$.

The inductive step is as follows. Let $\Gamma$ be a graph with $b_1(\Gamma)=n$ and choose $e\in E(\Gamma).$ The hyperplane $h_e$ in $T^n$ corresponding to $e$ has on it the arrangement corresponding to $\Gamma\setminus e$. On the other hand, removing this hyperplane leaves us with the arrangement on $T^n$ corresponding to $\Gamma/e$.  Going back to the arrangement corresponding to $\Gamma$ by adding back in $h_e$, we note that each $n-1$ cell on $h_e$ intersects an $n$-cell of the arrangement corresponding to $\Gamma/e$, dividing it into two $n$-cells. The upshot is that $|\mathring{\text{Sk}}_n(P_\bullet(h_\Gamma^{-1}(0)))| = |\mathring{\text{Sk}}_n(P_\bullet(h_{\Gamma/e}^{-1}(0)))|+|\mathring{\text{Sk}}_n(P_\bullet(h_{\Gamma\setminus e}^{-1}(0)))| = t(\Gamma/e) + t(\Gamma\setminus e) = t(\Gamma)$, where the last equality comes about via the deletion-contraction relationship on the number of spanning trees.
\end{proof}

Let $\Gamma$ be a graph with no disconnecting vertices. Denote by $\hat\Gamma$ its \emph{base graph}, the graph we get by repeatedly contracting an edge in $\Gamma$ which is incident to a vertex of degree $2$. For a fixed first Betti number, there are a finite number of graphs with no degree $2$ vertices. 
Theorem \ref{smalldelcon} allows one to calculate the cohomology of any hypertoric Hitchin systems from knowledge of these base cases, formulated in the following corollary.

\begin{corollary}\label{intermsofbase}
Let $\Gamma$ be a graph with no disconnecting vertices, $\hat\Gamma$ its base graph, and $e_1,\dots, e_m$ an enumeration of the edges of $\hat\Gamma$. Further define a partition $\bigsqcup S_i$ of the set of degree $2$ vertices in $\Gamma$ via the relation $v\in S_i$ if there is a path from $v$ to a vertex adjacent to $e_i$ which goes through only degree $2$ vertices. Then
\begin{align*}
H^k(T^n,R^if_*\underline{\QQ}_\Gamma) &= H^k(T^n,R^if_*\underline{\QQ}_{\hat{\Gamma}}) \\&\qquad \oplus \bigoplus_{1\leq a_1<\dots<a_p\leq m} H^{k-p}(T^n,R^{i-p}f_*\underline{\QQ}_{\hat{\Gamma}\setminus e_1,\dots,e_p})^{\oplus\sum_{j=1}^p|S_{a_j}|}.
\end{align*}
\end{corollary} 

\begin{proof}
This statement follows from repeated applications of Theorem \ref{smalldelcon} and reorganization of the terms. 

Let us see some of the details. We write $S_i = \{e^i_1,\dots,e^i_{|S_{a_i}|}\}$ and suppress ``$T^n$" in each term. By repeatedly applying Theorem \ref{smalldelcon} to the $k$-th degree cohomology (and noting that $H^k(R^if_*\underline{\QQ}_{\Gamma \setminus e'})=H^k(R^if_*\underline{\QQ}_{\Gamma \setminus e''})$ for two degree $2$ edges in the same $S_i$), we find
\begin{align*}
H^k(R^if_*\underline{\QQ}_\Gamma) &= H^k(R^if_*\underline{\QQ}_{\Gamma / e^1_1}) \oplus H^{k-1}(R^{i-1}f_*\underline{\QQ}_{\Gamma\setminus e^1_1})\\
&= H^k(R^if_*\underline{\QQ}_{\Gamma / e^1_1, e^1_2}) \oplus H^{k-1}(R^{i-1}f_*\underline{\QQ}_{\Gamma\setminus e^1_1, e^1_2})^{\oplus 2}\\
&\vdots \\
&= H^k(R^if_*\underline{\QQ}_{\Gamma / S_1}) \oplus H^{k-1}(R^{i-1}f_*\underline{\QQ}_{\Gamma\setminus S_1})^{\oplus {|S_{1}|}}\\
&= H^k(R^if_*\underline{\QQ}_{(\Gamma / S_1)/e^2_1}) \oplus H^{k-1}(R^{i-1}f_*\underline{\QQ}_{(\Gamma / S_1)\setminus e^2_1})\\
&\qquad\oplus H^{k-1}(R^{i-1}f_*\underline{\QQ}_{\Gamma\setminus S_1})^{\oplus {|S_{a_1}|}}\\
&\vdots \\
&= H^k(R^if_*\underline{\QQ}_{\hat\Gamma}) \oplus \bigoplus_{a_1=1}^m H^{k-1}(R^{i-1}f_*\underline{\QQ}_{(\Gamma / S_1,\dots,S_{a_1-1})\setminus S_{a_1}})^{\oplus {|S_{a_1}|}}
\end{align*}

Repeating this process on the degree $k-1$ cohomology terms (and so on) yields the following, which only a mother could love:
\begin{align*}
H^k(R^if_*\underline{\QQ}_\Gamma) &= H^k(R^if_*\underline{\QQ}_{\hat\Gamma}) \oplus 
\bigoplus_{a_1=1}^m \Bigg(H^{k-1}(R^{i-1}f_*\underline{\QQ}_{\hat\Gamma\setminus e_{a_1}})\\
&\qquad \oplus \bigoplus_{a_2=a_1+1}^m \Bigg( H^{k-2}(R^{i-2}f_*\underline{\QQ}_{\hat\Gamma\setminus e_{a_1},e_{a_2}})\\
&\qquad \oplus \bigoplus_{a_3=a_2+1}^m \Bigg( H^{k-3}(R^{i-3}f_*\underline{\QQ}_{\hat\Gamma\setminus e_{a_1},e_{a_2},e_{a_3}})\\
&\qquad \oplus \bigoplus_{a_4=a_3+1}^m \Bigg( \dots \Bigg)^{\oplus |S_{a_3}|} \Bigg)^{\oplus |S_{a_2}|} \Bigg)^{\oplus {|S_{a_1}|}}
\end{align*}

To rearrange this into something manageable, observe that collecting the $k-p$ degree cohomology groups above yields exactly $$\bigoplus_{1\leq a_1<\dots<a_p\leq m} H^{k-p}(T^n,R^{i-p}f_*\underline{\QQ}_{\hat{\Gamma}\setminus e_1,\dots,e_p})^{\oplus\sum_{j=1}^p|S_{a_j}|}.$$
\end{proof}

The next couple of results cover the low-dimensional examples.

\begin{proposition}\label{dim1}
If $\Gamma$ is a graph with $b_1(\Gamma) = 1$, then the Poincar\'e polynomial of $\mathfrak{D}(\Gamma)$ is 
$$1+y+|E(\Gamma)|y^2.$$
\end{proposition}
\begin{proof}
A graph with $b_1(\Gamma) = 1$ (and no bridges) is just a cycle of $|E(\Gamma)|$ edges and so $h_\Gamma^{-1}(0)$ is a necklace of $|E(\Gamma)|$ copies of $\PP^1$, which has the stated Poincar\'e polynomial.
\end{proof}

\begin{theorem}\label{dim2}
Let $\Gamma$ be a graph with $b_1(\Gamma) = 2$ and $t(\Gamma)$ be the number of spanning trees of $\Gamma$. Then
\begin{enumerate}
\item If $\Gamma$ has a disconnecting vertex, then the Poincar\'e polynomial of $\mathfrak{D}(\Gamma)$ is 
$$1+2y+(|E(\Gamma)|+1)y^2+|E(\Gamma)|y^3+t(\Gamma)y^4.$$
\item If $\Gamma$ does not have a disconnecting vertex, then the Poincar\'e polynomial of $\mathfrak{D}(\Gamma)$ is 
$$1+2y+|E(\Gamma)|y^2+(|E(\Gamma)|-1)y^3+t(\Gamma)y^4.$$
\end{enumerate}

\end{theorem}
\begin{proof}
For part (1), note that Corollary \ref{disconcorr} implies that the Poincar\'e polynomial of $\frakD(\Gamma)$ is $\calP_y(\frakD(\Gamma)) = \calP_y(\frakD(\Gamma_1))\calP_y(\frakD(\Gamma_2))$
and Proposition \ref{dim1} simplifies this to  
\begin{align*}
\calP_y(\frakD(\Gamma))&=\left(1+y+|E(\Gamma_1)|y^2\right)\left(1+y+|E(\Gamma_2)|y^2\right)\\
&=1+2y+\left(|E(\Gamma_1)|+|E(\Gamma_2)|+1\right)y^2\\
&\qquad+\left(|E(\Gamma_1)|+|E(\Gamma_2)|\right)y^3+|E(\Gamma_1)||E(\Gamma_2)|y^4\\
&=1+2y+(|E(\Gamma)|+1)y^2+|E(\Gamma)|y^3+t(\Gamma)y^4
\end{align*}

Part (2) requires a bit more work. The idea is to first find the cohomology groups of $\frakD(\Theta_3)$ and then apply Theorem \ref{smalldelcon}, since any graph with first Betti number $2$ and no disconnecting vertices can be reduced to $\Theta = \triquiver$ via contractions of edges which are incident to vertices of degree $2$.

To calculate $H^\bullet(\frakD(\Theta_3),\underline{\QQ}_{\frakD(\Theta_3)} )$, we recall that thanks to Equation \eqref{cohomdecomp} it suffices to find the cohomology groups of the cell-compatible sheaves $R^if_*\underline{\QQ}_{\Theta_3}$ on $P_\bullet(h^{-1}_{\Theta_3}(0))$. In this case, $P_\bullet(h^{-1}_{\Theta_3}(0))$ has the topological type of a $2$-torus. So we find

\begin{equation*}
H^k(T^2,R^0f_*\underline{\QQ}_{\Theta_3}) = \begin{cases}
\QQ^{1} , &k=0,2\\
\QQ^{ 2} , &k=1\\
0, &\text{otherwise}
\end{cases}
\end{equation*}
trivially, and 
\begin{equation*}
H^k(T^2,R^2f_*\underline{\QQ}_{\Theta_3}) = \begin{cases}
\QQ^{ 3} , &k=2\\
0, &\text{otherwise}
\end{cases}
\end{equation*}
by Lemma \ref{gribby}.
Lastly, the calculation of the cohomology groups of $R^1f_*\underline{\QQ}_{\Theta_3}$ was conveniently placed earlier in this article as Example \ref{3edgecellularcomp}, where it was found that 
\begin{equation}\label{R1statement}
H^k(T^2,R^1f_*\underline{\QQ}_{\Theta_3}) = \begin{cases}
\QQ^{\oplus 2}, &k=1,2\\
0, &\text{otherwise}
\end{cases}
\end{equation}
Now let us prove by induction that
\begin{equation*}
H^k(T^2,R^1f_*\underline{\QQ}_\Gamma) = \begin{cases}
\QQ^{|E(\Gamma)|-1}, &k=1,2\\
0, &\text{otherwise}
\end{cases}
\end{equation*}
for $\Gamma$ a graph with $b_1(\Gamma)=2$ and no disconnecting vertices. Equation \eqref{R1statement} serves as our base case.  Assuming that the statement holds for a graph with $k$ edges, let $\Gamma$ be a graph with $|E(\Gamma)|=k+1$ and $e$ an edge of $\Gamma$ which is adjacent to a vertex of degree $2$. Then Theorem \ref{smalldelcon} yields
\begin{equation*}
\begin{split}
H^1(T^2,R^1f_*\underline{\QQ}_\Gamma)&\cong H^1(T^2,R^1f_*\underline{\QQ}_{\Gamma/e})\oplus H^0(T^2,R^0f_*\underline{\QQ}_{\Gamma\setminus e})\\
&\cong \QQ^{|E(\Gamma/e)|-1}\oplus \QQ \\
&\cong \QQ^{|E(\Gamma)|-2}\oplus \QQ \\
&\cong \QQ^{|E(\Gamma)|-1} 
\end{split}
\end{equation*}
and similarly for $H^2(T^2,R^1f_*\underline{\QQ}_\Gamma)$, so we are done, in view of the fact that Lemma \ref{gribby} again yields 
\begin{equation*}
H^k(T^2,R^2f_*\underline{\QQ}_\Gamma) = \begin{cases}
\QQ^{|\mathring{\text{Sk}}_2|} , &k=2\\
0, &\text{otherwise}
\end{cases}
\end{equation*}
and $|\mathring{\text{Sk}}_2|=t(\Gamma)$ by Lemma \ref{tree}.
\end{proof}

We leave as an exercise the analogous formulas for higher dimensions, and would be especially happy to see such a formula for graphs with $b_1(\Gamma)=3$ (hint: there are four graphs with $b_1(\Gamma)=3$, no disconnecting vertices, and no vertices of degree $2$ to which one should apply the methods of Section \ref{section:cellular}).

\bibliographystyle{acm} 
\bibliography{hths}

\end{document}